\documentclass{amsart}

\usepackage{esint}
\usepackage{xstring}
\usepackage{enumerate}
\usepackage{textgreek}
\usepackage{amssymb}
\usepackage{hyperref}
\usepackage{graphicx}
\usepackage{mathrsfs}

\makeatletter
\@namedef{subjclassname@2020}{%
  \textup{2020} Mathematics Subject Classification}
\makeatother

\newtheorem{theorem}{Theorem}[section]
\newtheorem{lemma}[theorem]{Lemma}
\newtheorem{proposition}[theorem]{Proposition}
\newtheorem{corollary}[theorem]{Corollary}

\theoremstyle{definition}
\newtheorem{definition}[theorem]{Definition}

\theoremstyle{remark}

\newcommand{\nor}[1]{\left\lVert #1 \right\rVert}
\newcommand{\nors}[1]{\lVert #1 \rVert}
\newcommand{\abs}[1]{\left\lvert #1 \right\rvert}
\newcommand{\pth}[1]{\left( #1 \right)}

\newcommand{\set}[1]{\left\lbrace #1 \right\rbrace}
\newcommand{\abset}[1]{\abs{\set{#1}}}

\newcommand{\pp}[1]{^{(#1)}}

\newcommand{\cnt}[3]{ _{#1 = #2} ^{\IfStrEq{#3}{i}{\infty}{#3}}}
\newcommand{\inv}{^{-1}}

\newcommand{\T}{\mathbb T}
\newcommand{\unu}{u ^\nu}
\newcommand{\unup}{u ^\nu _{\operatorname{Pr}}}

\newcommand{\Unu}{U ^\nu}
\newcommand{\Unup}{U ^\nu _{\operatorname{Pr}}}
\newcommand{\Pnu}{P ^\nu}
\newcommand{\ub}{\bar u}
\newcommand{\Ub}{\bar U}

\newcommand{\Pb}{\bar P}

\newcommand{\Bp}[1]{B ^+ _#1}
\newcommand{\Bb}[1]{\bar B _#1}

\newcommand{\Qp}[1]{Q ^+ _#1}
\newcommand{\Qb}[1]{\bar Q _#1}

\newcommand{\Qbi}{\bar Q ^i}
\newcommand{\Bbi}{\bar B ^i}

\newcommand{\mQ}[1]{\mathcal Q _#1}
\newcommand{\mQo}{\mathcal Q ^\circ}

\renewcommand{\Re}{\mathsf{Re}}
\newcommand{\per}{_{\text{per}}}

\newcommand{\pOmega}{\partial \Omega}
\newcommand{\tomega}{\tilde \omega}
\newcommand{\tomeganu}{\tomega ^\nu}
\newcommand{\indtomega}{\indWithSet{\tomega > \maxs{\frac1t, \frac1{W ^2}, \frac1{H ^2}}}}

\newcommand{\indtomeganu}{\indWithSet{\nu \tomega ^\nu > \maxs{\frac\nu t, \frac{\nu ^2}{W ^2}, \frac{\nu ^2}{H ^2}}}}

\newcommand{\qb}{\bar q}
\newcommand{\eb}{\bar e}
\newcommand{\fb}{\bar f}
\newcommand{\vb}{\bar v}

\newcommand{\tOmega}{\tilde \Omega}

\newcommand{\R}{\mathbb{R}}
\newcommand{\RR}[1]{\R ^{#1}}
\newcommand{\Rd}{{\R ^d}}
\newcommand{\pt}{\partial _t}

\newcommand{\grad}{\nabla}
\newcommand{\La}{\Delta}
\renewcommand{\div}{\operatorname{div}}
\newcommand{\curl}{\operatorname{curl}}
\newcommand{\tensor}{\otimes}

\newcommand{\inn}{\text{ in }}
\newcommand{\onn}{\text{ on }}
\newcommand{\att}{\text{ at }}
\renewcommand{\ae}{\text{ a.e. }}

\newcommand{\forr}{\text{ for }}

\renewcommand*{\d}{\mathop{\kern0pt\mathrm{d}}\!{}}
\newcommand{\dt}{\d t}
\newcommand{\dx}{\d x}

\newcommand{\pfr}[2]{\frac{\partial #1}{\partial #2}}
\newcommand{\dfr}[2]{\frac{\d #1}{\d #2}}
\newcommand{\ddt}{\dfr{}t}
\newcommand{\pthf}[2]{\pth{\frac{#1}{#2}}}

\newcommand{\at}[1]{\bigr\rvert _{#1}}

\newcommand{\half}{\frac12}

\newcommand{\nmL}[2]{\nor{#2} _{L ^{#1}}}
\newcommand{\nmLW}[2]{\nor{#2} _{L ^{#1} (\Omega)}}
\newcommand{\nmLWT}[2]{\nor{#2} _{L ^{#1} ((0, T) \times \Omega)}}
\newcommand{\nmLWt}[2]{\nor{#2} _{L ^{#1} ((0, t) \times \Omega)}}
\newcommand{\nmLpW}[2]{\nor{#2} _{L ^{#1} (\partial \Omega)}}
\newcommand{\nmLpWT}[2]{\nor{#2} _{L ^{#1} ((0, T) \times \partial \Omega)}}

\newcommand{\nmLtx}[4]{\nor{#4} _{L ^{#1} _t L ^{#2} _x (#3)}}
\newcommand{\nmLpqIS}[5]{\nor{#5} _{L ^{#1} (#2; L ^{#3} (#4))}}

\newcommand{\Lt}[1]{L _t ^{#1}}
\newcommand{\Lx}[1]{L _x ^{#1}}
\newcommand{\Lz}[1]{L _z ^{#1}}
\newcommand{\Lpq}[4]{L ^{#1} (#3; L ^{#2} (#4))}

\newcommand{\nmpq}[3]{\nor{#3} _{L ^{#1}(0, T; L ^{#2} (\Omega))}}

\newcommand{\indWithSet}[1]{\mathbf1_{\set{#1}}}
\newcommand{\ind}[1]{\mathbf1_{#1}}
\newcommand{\intset}[1]{\int _{\set{#1}}}

\newcommand{\mins}[2][]{\min _{#1} \set{#2}}
\newcommand{\maxs}[2][]{\max _{#1} \set{#2}}

\newcommand{\e}{\varepsilon}
\newcommand{\mm}{\mathcal M}

\newcommand{\loc}{\mathrm{loc}}
\newcommand{\Id}{\mathrm{Id}}
\newcommand{\osc}{\mathrm{osc}}
\newcommand{\hfsq}[1]{\frac{\abs{#1} ^2}{2}}

\newcommand{\pz}{\partial _z}
\newcommand{\rstar}{r _\star}

\newcommand{\ulim}{u ^{\infty}}
\renewcommand{\overline}[1]{\operatorname{closure} \big\lbrace #1 \big\rbrace}
\begin{document}

\title[Boundary Vorticity of Navier-Stokes and Inviscid Limit]{Boundary Vorticity Estimates for Navier-Stokes and Application to the Inviscid Limit}

\author{Alexis F. Vasseur}
\address{Department of Mathematics, 
The University of Texas at Austin, 
2515 Speedway Stop C1200,
Austin, TX 78712, USA}
\email{vasseur@math.utexas.edu}

\author{Jincheng Yang}
\address{Department of Mathematics, 
The University of Texas at Austin,
2515 Speedway Stop C1200, 
Austin, TX 78712, USA}
\email{jcyang@math.utexas.edu}

\date{\today}

\thanks{\textit{Acknowledgment}. The first author was partially supported by the NSF grant: DMS 1907981. The second author was partially supported by the NSF grant: DMS-RTG 1840314.}

\begin{abstract}
    Consider the steady solution to the incompressible Euler equation $\bar u = A e _1$ in the periodic tunnel $\Omega = \mathbb T ^{d - 1} \times (0, 1)$ in dimension $d = 2, 3$. Consider now the family of solutions $u ^\nu$ to the associated Navier-Stokes equation with the no-slip condition on the flat boundaries, for small viscosities $\nu = A / \mathsf{Re}$, and initial values in $L ^2$. We are interested in the weak inviscid limits up to subsequences $u ^\nu \rightharpoonup u ^\infty$ when both the viscosity $\nu$ converges to 0, and the initial value $u ^\nu _0$ converges to $A e _1$ in $L ^2$. Under a conditional assumption on the energy dissipation close to the boundary, Kato showed in 1984 that $u ^\nu$ converges to $A e _1$ strongly in $L ^2$ uniformly in time under this double limit.
    It is still unknown whether this inviscid limit is unconditionally true. The convex integration method produces solutions $u _E$ to the Euler equation with the same initial values $A e_1$ which verify at time $0 < T < T _0$:
    $
        \lVert u _E (T) - A e _1 \rVert _{L ^2 (\Omega)} ^2 \approx A ^3 T.
    $
    This predicts the possibility of a layer separation with an energy of order $A^3 T$.
    We show in this paper that the energy of layer separation associated with any asymptotic $u ^\infty$ obtained via double limits cannot be more than
    $ 
        \lVert u ^\infty (T) - A e _1 \rVert _{L ^2 (\Omega)} ^2 \lesssim A ^3 T.
    $
    This result holds unconditionally for any weak limit of Leray-Hopf solutions of the Navier-Stokes equation.
    Especially, it shows that, even if the limit is not unique, the shear flow pattern is observable up to time $1 / A$. This provides a notion of stability despite the possible non-uniqueness of the limit predicted by the convex integration theory.
    The result relies on a new boundary vorticity estimate for the Navier-Stokes equation. This new estimate, inspired by previous work on higher regularity estimates for Navier-Stokes, provides a nonlinear control scalable through the inviscid limit.
\end{abstract}

\maketitle

\tableofcontents

\setlength{\parskip}{0.3em}

\section{Introduction}
For dimension $d=2, 3$, we consider the periodic channel with physical boundary at $x _d = 0$ and $x _d = 1$: $\Omega = \T ^{d - 1} \times (0, 1)$, where $\T = [0, 1] \per$ denotes the unit periodic domain.
For any kinematic viscosity $\nu >0$, we denote $\unu : (0, T) \times \Omega \to \Rd$ the velocity field of an incompressible fluid confined in $\Omega$, 
subject to no-slip boundary conditions, and  $\Pnu: (0, T) \times \Omega \to \R$ the associated pressure field. The dynamic of the flow is described by the following Navier-Stokes Equation:
\begin{align}
    \label{eqn:nse-nu}
    \tag{$\text{NSE} _\nu$}
    \begin{cases}
        \pt \unu + \unu \cdot \grad \unu + \grad \Pnu = \nu \La \unu & \inn (0, T) \times \Omega \\  
        \div \unu = 0 & \inn (0, T) \times \Omega \\
        \unu = 0 & \text{ for } 
        x_d=0,\text{ and }x_d=1.
    \end{cases}
\end{align}

For any $A>0$, we investigate the inviscid asymptotic behavior of $\unu$
when  $\nu$ converges to 0,  under the condition that the initial values converge to a shear flow of strength $A$:
\begin{equation}\label{ini-lim}
\lim _{\nu \to 0} \nmLW2{\unu (0) - A e _1} = 0.
\end{equation}
Note that the steady shear flow $\ub(t,x)=Ae_1$ is solution to the Euler equation with impermeability 
boundary condition:
\begin{align}
    \label{eqn:euler}
    \tag{EE}
    \begin{cases}
        \pt \ub + \ub \cdot \grad \ub + \grad \Pb = 0 & \inn (0, T) \times \Omega \\
        \div \ub = 0 & \inn (0, T) \times \Omega \\
        \ub \cdot n = 0 & \text{ for } x_d=0,\text{ and }x_d=1,
    \end{cases}
\end{align}
where $n$ is the outer normal as shown in Figure \ref{fig:channel}. 
However, it is an outstanding open question (even in dimension 2) whether, in the double limit \eqref{ini-lim} and  $\nu \to 0$, the solution $\unu$ of \eqref{eqn:nse-nu} converges to this shear flow  $A e_1$. The difficulty of this problem stems from the discrepancy between the no-slip boundary condition for the Navier-Stokes equation and the impermeable boundary condition of the Euler equation. Kato \cite{Kato1984} showed in 1984 a conditional result ensuring this convergence under the a priori assumption that the energy dissipation rate in a very thin boundary layer $\Gamma_\nu$ of width proportional to $\nu$ vanishes:
\[ 
    \lim _{\nu \to 0} \int _0 ^T \int _{\Gamma _\nu} \nu \abs{\grad \unu} ^2 \dx \dt = 0. 
\]
This condition has been sharpened in a variety of ways (see, for instance 
\cite{Temam1997, Wang2001, Kelliher2007, Kelliher2008} and  Kelliher \cite{Kelliher2017}, for a general review), and similar other conditional results have been derived (see for instance \cite{Bardos2012, Constantin2015, Constantin2017, Constantin2018}). Non-conditional results of strong inviscid limits have been obtained only for real analytic initial data \cite{Caflisch1998}, vanishing vorticity near the boundary \cite{Maekawa2014, Fei2018}, or symmetries \cite{Filho2008, Mazzucato2008}. 
Since \cite{Prandtl1904}, it is expected that in favorable cases, the Prandtl boundary layer describes the behavior of the solution $\unu$ up to a distance proportional to $\sqrt{\nu}$. However, even in the simple shear flow case, it is possible to engineer families of initial values $\unu(0)$ converging to the shear flow, but associated to Prandtl boundary layers which are either strongly unstable \cite{Grenier2000}, blow up in finite time \cite{E2000}, or even ill-posed in the Sobolev framework \cite{Varet2010, Varet2012}.

It is actually believed that the inviscid asymptotic limit may fail due to turbulence (See  Bardos and Titi \cite{Bardos2013}). This scenario is consistent with the non-uniqueness pathology of the shear flow solution for the Euler system  \eqref{eqn:euler}. 
Indeed, an adaptation to the boundary value problem  \eqref{eqn:euler} of the construction based on convex integration of Szekelyhidi in \cite{Szekelyhidi2011}   provides infinitely many solutions to \eqref{eqn:euler} with initial value $Ae_1$ (see also Bardos, Titi, Wiedemann \cite{Bardos2012} for a different boundary geometry). 
More precisely, the following estimate can be proved on this construction (see appendix \ref{appendix}).
\begin{proposition}\label{prop-convex}
For any $0 < C < 2$, there exists a solution $v$ to \eqref{eqn:euler} with initial value $Ae_1$ such that for any time $T < 1 / (2 A)$:
\[ 
    \nmLW2{v (T) - A e _1} ^2 = C A ^3 T.
\] 
\end{proposition}
The convex integration is a powerful tool introduced by De Lellis and Szekelyhidi \cite{DeLellis2009} to construct spurious solutions to the Euler equation. It proved itself to be a powerful tool to model turbulence. For instance, the technique was successfully applied by Isett \cite{Isett2018} to prove the Onsager theorem  (see also \cite{Buckmaster2019} for the construction of admissible solutions, and   \cite{Constantin1994} for the proof of the other direction). It shows that turbulent flows can have regularity $C^\alpha$ for any  $\alpha$ up to $1/3$, a property conjectured by Onsager \cite{Onsager1949}.  
Proposition \ref{prop-convex} predicts the possible deviation from the initial shear flow $Ae_1$ due to turbulence, a phenomenon called layer separation. Moreover, it provides an explicit value for the $L^2$ norm of this layer separation.  

This article aims to provide an upper bound on the $L^2$ norm of possible layer separations through the double limit inviscid asymptotic. 
In our channel framework, the Reynolds number is given by $\Re = A/\nu$.
Our main theorem is the following.

\begin{theorem}
    \label{thm:constant-shear}
    Let $\Omega$ be a unit periodic channel in $\Rd$ of dimension $d = 2, 3$. There exists $C > 0$ depending on $d$ only, such that the following is true.
    Let $\ub = A e _1$ be a constant shear flow for some $A > 0$, and let $\unu$ be a Leray-Hopf solution to \eqref{eqn:nse-nu} with kinematic viscosity $\nu > 0$.
    For any $T > 0$, we have
    \begin{align*}
        & \nmLW2{\unu (T) - \ub} ^2 + \frac\nu2 \nmLWT2{\grad \unu} ^2 \\ 
        & \qquad \le 4 \nmLW2{\unu (0) - \ub} ^2 + C A ^3 T + C A ^2 \Re \inv \log (2 + \Re).
    \end{align*}
\end{theorem}

This theorem is the special case of a more general result given in Theorem \ref{thm:main} at the end of this section.
By Leray-Hopf solution, we mean any weak solutions to \eqref{eqn:nse-nu} which in addition verifies the energy inequality:
$$
 \half \ddt \nmLW2{\unu} ^2\leq -\nu \nmLW2{\grad \unu} ^2.
$$
We have the following corollary on any weak inviscid limit, which corresponds to the layer separation predicted by Proposition \ref{prop-convex}.

\begin{corollary} 
\label{cor:weak-limit}
    There exists a universal constant $C > 0$ such that the following is true. Consider any family $\unu$ of a Leray-Hopf solutions to \eqref{eqn:nse-nu} such that $\unu _0$ converges strongly in $L ^2 (\Omega)$ to $A e _1$. Then, for any weak limit $\ulim$ of weakly convergent subsequences of $\unu$, we have for almost every $T > 0$ that 
    \begin{align*}
        \nmLW2{\ulim (T) - A e _1} ^2 \le C A ^3 T.
    \end{align*}
\end{corollary}

Note that the solutions $\unu$ are uniformly bounded in $L ^\infty (\R ^+, L^2 (\Omega))$. Therefore they converge weakly up to a subsequence in $L ^2 _{t, x}$.

\vskip0.3cm

This result bets on the fact that the double limit to $A e _1$ in the inviscid asymptotic may fail, which is related to the physical relevance of the solutions constructed by convex integration. An interesting question is whether such solutions can be themselves obtained via double limit in the inviscid asymptotic. A first result in this direction was provided by Buckmaster and Vicol \cite{BuckmasterVicol2019} where they constructed via convex integration, in the case without boundary, spurious solutions at the level of Navier-Stokes. They show that the inviscid limit of this family of Navier-Stokes solutions can converge to spurious solutions of Euler. However, these spurious solutions constructed at the level of Navier-Stokes do not have enough regularity to be Leray-Hopf solutions, and therefore do not fit in the framework of Corollary \ref{cor:weak-limit}.



\paragraph*{Non-uniqueness and pattern predictability} 

The non-uniqueness of solutions to the Euler equation, as proved by convex integration, puts under question the ability of the model itself to predict the future.  
Theorem \ref{thm:constant-shear} provides a first example of how non-uniqueness and pattern predictability can be reconciled. The energy of the shear flow is $A ^2$, while the maximum energy of the layer separation is bounded above by $C A^3 T$. This predicts pattern visibility on a lapse of time $1/A$. On this lapse of time, the layer separation stays negligible compared to the shear flow pattern. Especially, the smaller the pattern is (small $A$), the longer the prediction stays accurate. 

\paragraph*{Inviscid limit and boundary vorticity} 
It is well known that the possible growth of the layer separation is closely related to the creation of boundary vorticity (see Kelliher \cite{Kelliher2007} for instance). 
To see this, 
we formally compute the evolution of the $L ^2$ distance between $\unu$ and $\ub$:
\begin{equation}
    \label{eqn:first-computation}
    \begin{aligned}
        \half \ddt \nmL2{\unu - \ub} ^2 &= (\unu - \ub, \pt \unu) \\
        &= -(\unu - \ub, \unu \cdot \grad \unu) - (\unu - \ub, \grad \Pnu) + \nu (\unu - \ub, \La \unu) \\
        &= \nu (\unu, \La \unu) - \nu (\ub, \La \unu) \\
        &= -\nu \nmL2{\grad \unu} ^2 - \int _{\pOmega} J[\ub] \cdot (\nu \omega ^\nu) \d x'
    \end{aligned}
\end{equation}
where $J [\ub] = n ^\perp \cdot \ub$ when $d = 2$ and $J[\ub] = n \times \ub$ when $d = 3$, and $\omega ^\nu$ is the vorticity of $\unu$. Since $\ub$ is a constant on the boundaries, it is crucial to estimate the mean boundary vorticity. If the convergence $\nu \omega ^\nu \at{\pOmega} \to 0$ holds in the average sense, then the inviscid limit would be valid. For a general static smooth solution to Euler's equation $\ub$ in a general domain $\Omega$, we only need $\nu \omega ^\nu \at{\pOmega} \to 0$ in distribution.
This convergence may fail and we could lose uniqueness, but we can still control the size of the impact from this boundary vorticity using Theorem \ref{thm:boundary-regularity} below.

\begin{figure}[htbp]
    \centering
    \includegraphics{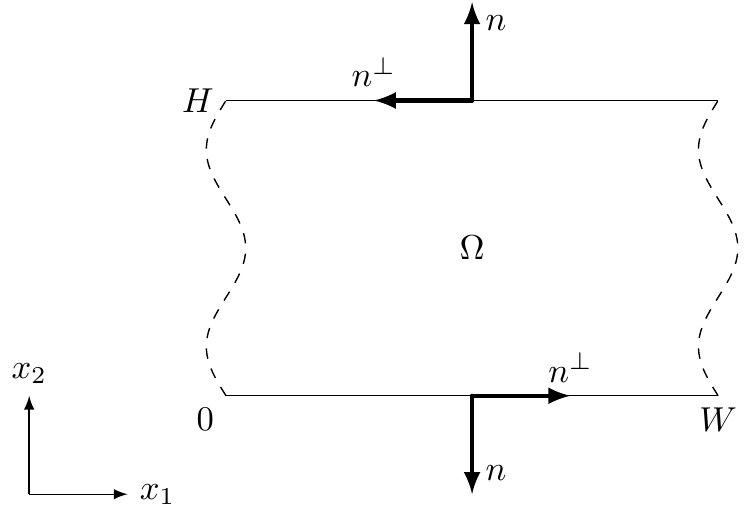}
    \caption{2D Periodic Channel}
    \label{fig:channel}
\end{figure}

Before showing the theorem, we first illustrate which estimates we may expect and how they prove Theorem \ref{thm:constant-shear}. Denote the energy dissipation by
\begin{align*}
    D := \nu \nmLWT2{\grad \unu} ^2.
\end{align*}
If we take the curl of \eqref{eqn:nse-nu}, we have the vorticity equation, 
\begin{align*}
    \pt \omega + u \cdot \grad \omega = \nu \La \omega + \omega \cdot \grad u.
\end{align*}
The main difficulties are due to the 
transport term $u \cdot \grad \omega$, and the boundary. Let us put aside those two difficulties for now, and focus on the other terms.  Then the regularity we could expect for $\omega$ is at best
\begin{align*}
    \nu ^2 \nmLWT1{\grad ^2 \omega} \lesssim _d \nu \nmLWT1{\omega \cdot \grad u} \le D.
\end{align*}
Here $A \lesssim _d B$ means $A \le C (d) B$ for some constant $C (d)$ depending in dimension $d$ only.
This is not rigorous because the parabolic regularization is false in $L ^1$, but let us also ignore this issue for the moment. By interpolation, we have 
\begin{align*}
    \nu ^\frac32 \nmLWT{\frac32}{\grad ^\frac23 \omega} ^\frac32 \lesssim _d \pth{
        \nu ^2 \nmLWT1{\grad ^2 \omega} 
    } ^\half \pth{
        \nu \nmLWT2{\omega} ^2 
    } ^\half \lesssim _d D.
\end{align*}
Finally the trace theorem suggests that (again, this is the borderline case for the trace theorem, so in no way a rigorous proof)
\begin{align}
    \label{eqn:formal-3/2}
    \nmLpWT{\frac32}{\nu \omega} ^\frac32 \lesssim _d D.
\end{align}
Using this $L ^{\frac32}$ estimate, if we integrate \eqref{eqn:first-computation} from $0$ to $T$, we have 
\begin{align*}
    & \half \nmLW2{\unu - \ub} ^2 (T) + D \\
    & \qquad \le \half \nmLW2{\unu - \ub} ^2 (0) + \nmLpWT1{J[\ub] \cdot \nu \omega ^\nu} \\
    & \qquad \le \half \nmLW2{\unu - \ub} ^2 (0) + 
    \nmLpWT{\frac32}{\nu \omega ^\nu} 
    \nmLpWT3{\ub} \\ 
    & \qquad \le \half \nmLW2{\unu - \ub} ^2 (0) + 
    \half D + C A ^3 T \abs{\pOmega} 
\end{align*}
for some constant $C$ depending on $d$ only. By absorbing $\half D$ to the left we finish the proof of Theorem \ref{thm:constant-shear}. Note however, that this direct proof collapses due to the transport term. In dimension three, $u$ can be controlled at best in $L^{10/3}_{t,x}$ while the best control of $\nabla \omega$ is in the Lorentz spaces $L^{4/3,q}_{t,x}$ for any $q>4/3$ (see \cite{Vasseur2021}). But this is far from enough to bound the transport term $u\nabla \omega$ in $L^1_{t,x}$. In dimension 2, the transport term can almost be controlled in $L^1$. But the bound is in negative power of $\nu$ and so is useless for the asymptotic limit. However, we can use blow-up techniques inspired by \cite{Vasseur2010} (see also \cite{Choi2014, Vasseur2021}) which naturally deplete the strength of the transport term. 

\paragraph*{Boundary vorticity control for the unscaled Navier-Stokes equation}

In the review paper \cite{Maekawa2018}, Maekawa and Mazzucato summarized the difficulties of   considering inviscid limit with boundary:
\begin{quote}
Mathematically, the main difficulty in the case of the no-slip boundary condition is the lack of a priori estimates on strong enough norms to pass to the limit, which in turn is due to the lack of a useful boundary condition for vorticity or pressure.     
\end{quote}
Following this remark, our proof relies on a new boundary vorticity control. This is a regularization result for the unscaled Navier-Stokes equation. However, it is remarkable that this estimate is rescalable through the inviscid limit $\nu\to0$. The strategy of looking for uniform estimates with respect to the inviscid scaling was first introduced for 1D conservation laws in \cite{Kang2021}. It was successfully applied to obtain the unconditional double limit inviscid asymptotic in the case of a single shock  \cite{Kang2021Inven}. 
Note that if $(\unu, P ^\nu)$ is a solution to \eqref{eqn:nse-nu}, then $u (t, x) = \unu (\nu t, \nu x)$, $P (t, x) = P ^\nu (\nu t, \nu x)$ solves the Navier-Stokes equation with unit viscosity coefficient in $(0, T / \nu) \times (\Omega / \nu)$:
\begin{align}
    \label{eqn:nse}
    \tag{$\text{NSE}$}
        \pt u + u \cdot \grad u + \grad P = \La u 
        , \qquad
        \div u = 0 
        .
\end{align}
The regularization result on the vorticity at the boundary is as follows. 

\begin{theorem}[Boundary Regularity]
    \label{thm:boundary-regularity}
    There exists a universal constant $C > 0$ such that the following holds.
    Let $\Omega$ be a periodic channel of period $W$ and height $H$ of dimension $d = 2$ or $3$. 
    For any Leray-Hopf solution $u$ to {\upshape (\hyperref[eqn:nse-nu]{$\text{NSE}_1$})} in $(0, T) \times \Omega$, there exists a parabolic dyadic decomposition \footnote{A dyadic decomposition into cubes of parabolic scaling. See Definition \ref{def:parabolic-dyadic-decomposition}.} 
    \[
        \overline{(0, T) \times \pOmega} = \overline{
            \bigcup _{i} \, (s ^i, t ^i) \times \bar B _{r _i} (x ^i)
        },
    \]
    where $0 \le s ^i < t ^i \le T$, $0 < r ^i < \frac W2$, $x ^i \in \pOmega$, and 
    \[
        \bar B _{r} (y) = \set{(x', x _d) \in \pOmega: \nor {x' - y'} _{\ell ^\infty} < r, x _d = y _d}
    \] 
    is a box of dimension $d - 1$ in $\pOmega$, such that the following is true. Define a piecewise constant function $\tomega: (0, T) \times \pOmega \to \R$ by taking averages
    \[
        \tomega (t, x) = \frac1{\abs{\bar B _{r ^i}}} \int _{\bar B _{r ^i} (x ^i)} \abs{
            \frac1{t ^i - s ^i} \int _{s ^i} ^{t ^i} \omega \dt
        } \dx', \qquad \forr t \in (s ^i, t ^i), x \in \bar B _{r ^i} (x ^i).
    \]
    Then 
    \begin{align*}
        \nmLpWT{\frac32, \infty}{\tomega \indtomega} ^{\frac32} \le C \nmLWT2{\grad u} ^2.
    \end{align*}
\end{theorem}

This theorem provides a ``scaling invariant'' nonlinear estimate, that is, both sides of the estimate have the same scaling under the canonical scaling of the Navier-Stokes equation $(t, x) \mapsto \e u (\e ^2 t, \e x)$. The bounds in the theorem do not depend on the size of $\Omega$ or the terminal time $T$, and we do not require any smallness for the initial energy.

The conclusion of this theorem is slightly different from what we hope in \eqref{eqn:formal-3/2}, due to some difficulties that we overlooked in the formal argument. To begin with, the higher regularity $\grad ^2 \omega \in L ^1$ is not known. As mentioned before, one reason is the transport term $u \cdot \grad \omega$ is indeed hard to control. Using blow-up techniques along the trajectories of the flow first introduced in \cite{Vasseur2010}, it was proved in \cite{Vasseur2021} that without boundary in $\Omega = \RR3$, $\grad ^2 \omega \in L ^{1,q}$ locally for $q > 1$ but miss the endpoint $L ^1$. The bounded domain is even more complicated because of the lack of convenient global control on the pressure. In turn, it means that no control on the pressure can be brought locally through the blow-up process. This poses problems when applying the boundary regularity theory for the linear evolutionary Stokes equation. Indeed, a counterexample constructed in \cite{Seregin2014} shows that we cannot control that way oscillations in time. 
The idea which remedies this problem consists in smoothing locally in time to gain some integrability. We can then apply the boundary Stokes estimate for $\int u \dt$ instead of $u$. This justifies the construction of $\tilde\omega$ via local smoothing in Theorem \ref{thm:boundary-regularity}.
Lastly, because the maximal function is not a bounded operator in $L ^1$, we only obtained weak $L ^\frac32$ norm instead of $L ^\frac32$ norm.

Note that because $J[\ub]$ is constant on the boundary $\partial \Omega$, and because $\tilde{\omega}$ is constructed via local smoothing on disjoint domains, we have 
\[
    \abs{
        \int _0 ^T \int_{\pOmega} J[\ub]\cdot \omega ^\nu \dx' \dt
    } \le \abs{
        \int _0 ^T \int_{\pOmega} J[\ub] \cdot \tomeganu \dx' \dt
    }.
\]
We can then apply Theorem \ref{thm:boundary-regularity},
and proceed as in the formal computation. One last difficulty is that Theorem \ref{thm:boundary-regularity} is a regularization result, and so the estimate weakens when $t$ goes to 0. Indeed, it controls only $\tomega > \maxs{\frac1t, \frac1{W ^2}, \frac1{H ^2}}$. If we integrate the remainder, there will be a logarithmic singularity at $t = 0$. To avoid this, we apply the vorticity bound only in the time interval $t \in (T _\nu, T)$ for some small time $T _\nu \approx \nu ^3$, and for $t \in (0, T _\nu)$ we use a very short time stability of a stable Prandtl layer to bridge the gap. 

\paragraph*{General case}

We actually  do the proof in a slightly more general setting. We will consider a periodic channel 
with width $W$ and  height $H$, where the physical boundary are localized at $x_d=0$ and $x_d=H$ (see Figure \ref{fig:channel}):
\begin{align*}
    \Omega = \set{
        (x', x _d): 0 \le x _d \le H, x' \in 
        [0, W] \per ^ {d - 1}
    }.
\end{align*}
The following theorem estimates the layer separation for a more general shear flow $\ub$ of the following form:
\[
    \ub (x) = \begin{cases}
        \Ub (x _2) e _1 & d = 2 \\
        \Ub _1 (x _3) e _1 + \Ub _2 (x _3) e _2 & d = 3 \quad .
    \end{cases}
\]
In this configuration, we define the Reynolds number as $$\Re = \frac{A H}{\nu}$$ where $A = \nmLpW{\infty}{\ub}$ is the boundary shear.

\begin{theorem}[General Shear Flow]
    \label{thm:main}
    There exists a universal constant $C > 0$ such that the following holds. Let $\Omega$ be a bounded periodic channel with period $W$ and height $H$ in $\Rd$ with $d = 2$ or $3$. Let $\ub$ be a static shear flow in $\Omega$ with bounded vorticity, and let $\unu$ be a Leray-Hopf solution to \eqref{eqn:nse-nu}. For a given $\ub$ defined as above, denote the maximum shear, boundary velocity, and kinetic energy of $\ub$ by 
    \begin{align*}
        G &:= \nmLW{\infty}{\grad \ub}, 
        &
        A &:= \nmLpW{\infty}\ub, 
        &
        E &:= \nmLW2{\ub} ^2.    
    \end{align*}
    For any $T > 0$, we have  
    \begin{align*}
        & \sup _{0 \le t \le T} \set{
            \nmLW2{\unu - \ub} ^2 (t) + \frac\nu2 \nmLWt2{\grad \unu} ^2
        } \\
        & \qquad \le \exp(2 G T) 
        \Biggl\{
            4 \nmLW2{\unu (0) - \ub} ^2 + 
            2 \nu G ^2 T \abs{\Omega} + C A ^2 \abs{\Omega} \Re \inv \log \pth{2 + \Re} \\
        &\qquad \qquad \qquad \qquad \quad + 2 \Re \inv E + C A ^3 T \abs{\pOmega} \maxs{H / W, 1} ^{2}
        \Biggr\}.
    \end{align*}
\end{theorem}

Note that Theorem \ref{thm:constant-shear} is a direct consequence of Theorem \ref{thm:main} with $H=W=1$, $\Ub=A$ for $d=2$, and $\Ub_1=A, \Ub_2=0$ for $d=3$.

This paper is organized as follows. We first introduce necessary tools in Section \ref{sec:prelimilary}. The boundary vorticity estimate and the proof of Theorem \ref{thm:boundary-regularity} is shown in Section \ref{sec:boundary}. In Section \ref{sec:main} we finish the proof of the main result, which are Theorem \ref{thm:constant-shear} and Theorem \ref{thm:main}. Finally, we prove Proposition \ref{prop-convex} in the appendix.

\section{Notations and Preliminary}
\label{sec:prelimilary}

We begin with some notations. 
We will be working with boxes more often than balls. For this reason, let us denote the spatial box and the space-time cube of radius $r$ by
\begin{align*}
    B _r &:= \set{x \in \Rd: \nor{x} _{\ell ^\infty} < r}, & 
    Q _r &:= (-r ^2, 0) \times B _r.
\end{align*}
We denote the same box and cube centered at $x$ and $(t, x)$ by $B _r (x)$ and $Q _r (t, x)$ respectively. Near the boundary $\set{x _d = 0}$, we denote the half-box and its boundary part by 
\begin{align*}
    \Bp r &:= \set{(x', x _d): \nors{x '} _{\ell ^\infty} < r, 0 < x _d < r}, &
    \Bb r &:= \set{(x', 0): \nors{x '} _{\ell ^\infty} < r},
\end{align*}
and denote their space-time version by
\begin{align*}
    \Qp r &= (-r ^2, 0) \times \Bp r, 
    &
    \Qb r &= (-r ^2, 0) \times \Bb r. 
\end{align*}
Finally, for a bounded set $\Omega$ and $f \in L ^2 (\Omega)$, we denote the average of $f$ in $\Omega$ as 
\[
    \fint _\Omega f \dx = \frac1{\abs{\Omega}} \int _\Omega f \dx.
\]

In this section, we provide some useful preliminary results and some corollaries, which will be used later in the paper. Most are widely known, and we do not claim any originality in the proof, but we include them here for completeness.

\subsection{Evolutionary Stokes Equation} 

Let $(u, P)$ be the solution to the following Stokes equation.
\begin{align}
    \label{eqn:stokes}
    \tag{SE}
    \begin{cases}
        \pt u + \grad P = \La u + f & \inn (0, T) \times \Omega \\
        \div u = 0 & \inn (0, T) \times \Omega
    \end{cases}.
\end{align}
Recall the following estimates on Stokes equations,
which can be found in the book of Seregin \cite{Seregin2014}.

\begin{theorem}[Cauchy Problem, Section {4.4} Theorem 4.5]
    \label{thm:cauchy}
    Let $\Omega$ be a bounded domain with smooth boundary. Let $1 < p, q < \infty$, and $f \in \Lpq pq{0,T}{\Omega}$. There exists a unique solution $(u, P)$ to \eqref{eqn:stokes} such that 
    \begin{enumerate}[\upshape (1)]
        \item $u$ satisfies the zero initial-boundary condition: 
        \begin{align*}
            u &= 0 \att t = 0, \\
            u &= 0 \onn (0, T) \times \pOmega.
        \end{align*}

        \item $P$ satisfies the zero mean condition: 
        \begin{align*}
            \int _{\Omega} P (t, x) \dx = 0 \text{ at any } t \in (0, T).
        \end{align*}
        
    \end{enumerate}
    Moreover, we have the coercive estimate 
    \begin{align*}
        \nmpq pq{\vert\pt u\vert + \vert\grad ^2 u\vert + \vert\grad P\vert} \le C (\Omega, p, q) \nmpq pqf.
    \end{align*}
\end{theorem}

\begin{theorem}[Local Boundary Regularity, Section {7.10} Proposition 7.10]
    \label{thm:boundary}
    Let $1 < p < \infty$, $1 < q \le q' < \infty$. Assume $u, \grad u, P \in \Lt p \Lx q (\Qp2)$, $f \in \Lt p \Lx {q'} (\Qp2)$ and $(u, P)$ satisfy \eqref{eqn:stokes} in $\Omega = \Qp2$. Moreover, assume 
    \begin{align}
        \label{eqn:flat}
        u = 0 \onn \{x _d = 0\}.
    \end{align}
    Then we have the local boundary estimate
    \begin{align*}
        & \nmLtx{p}{q'}{\Qp1}{
            \vert\pt u\vert + \vert\grad ^2 u\vert + \vert\grad P\vert
        } \\
        & \qquad \le C (p, q, q') \pth{
            \nmLtx pq{\Qp2}{
                \vert u\vert + \vert\grad u\vert + \vert P\vert
            } 
            +
            \nmLtx{p}{q'}{\Qp1}{f}
        }.
    \end{align*}
\end{theorem}

Combining these two estimates, we derive the following mixed case.
\begin{corollary}
    \label{cor:mixed}
    Let $1 < p _2 < p _1 < \infty$, $1 < q _1, q _2 < \infty$, $f \in \Lt {p _1} \Lx {q _1} (\Qp2)$, $u, \grad u, P \in \Lt {p _2} \Lx {q _2} (\Qp2)$. If $(u, P)$ satisfies \eqref{eqn:stokes} in $\Qp2$ and $u$ satisfies \eqref{eqn:flat}, then $u = u _1 + u _2$ satisfying for any $q' < \infty$, there exists a constant $C = C(p _1, p _2, q _1, q _2, q')$ such that 
    \begin{align*}
        & \nmLtx{p _1}{q _1}{\Qp1}{
            \vert\pt u _1\vert + \vert\grad ^2 u _1\vert
        }
        + 
        \nmLtx{p _2}{q '}{\Qp1}{
            \vert\pt u _2\vert + \vert\grad ^2 u _2\vert
        } \\
        & \qquad \le 
        C \pth{
            \nmLtx{p _1}{q _1}{\Qp2}{f}+ \nmLtx{p _2}{q _2}{\Qp2}{
                \vert u\vert + \vert\grad u\vert + \vert P\vert
            }
        }.
    \end{align*}
\end{corollary}

\begin{proof}
    Let $\Omega'$ be a smooth domain such that $B ^+ _{\frac32} \subset \Omega' \subset B ^+ _2$. Define $u _1$ to be the solution to the Cauchy problem in $\Omega'$ with force $f$. By Theorem \ref{thm:cauchy}, we obtain
    \begin{align*}
        \nmLpqIS{p _1}{-4, 0}{q _1}{\Omega '}{
            \vert\pt u _1\vert + \vert\grad ^2 u _1\vert + \vert\grad P _1\vert
        } \le C \nmLtx{p _1}{q _1}{\Qp2}{f}.
    \end{align*}
    Since $u _1$ has trace zero, $P _1$ has mean zero, we have 
    \begin{align*}
        \nmLpqIS{p _1}{-4, 0}{q _1}{\Omega '}{
            \vert u _1\vert + \vert\grad u _1\vert + \vert P _1\vert
        } \le C \nmLtx{p _1}{q _1}{\Qp2}{f}.
    \end{align*}
    Now we define $u _2 = u - u _1, P _2 = P - P _1$. Since $p _1 > p _2$, we have 
    \begin{align*}
        & \nmLtx{p _2}{\mins{q _1, q _2}}{\Qp{{3/2}}}{
            \vert u _2\vert + \vert\grad u _2\vert + \vert P _2\vert
        } \\
        & \qquad \le C \pth{
            \nmLtx{p _1}{q _1}{\Qp2}{f} + \nmLtx{p _2}{q _2}{\Qp2}{
                \vert u\vert + \vert\grad u\vert + \vert P\vert
            }
        }.
    \end{align*}
    Note that $u _2$ solves \eqref{eqn:stokes} with zero force term in $\Qp\frac32$,
    so the desired result follows by applying Theorem \ref{thm:boundary}.
\end{proof}

\subsection{Inhomogeneous Sobolev Embedding}

We show that given partial derivatives bounded in inhomogeneous Lebesgue spaces, a binary function is bounded in $L ^\infty$.

\begin{lemma}[Inhomogeneous Supercritical Sobolev Embedding]
    \label{lem:sob}
    Let $\alpha \in (0, 1)$, and $\Omega = \set{(t, z): t \in [-1, 0], z \in [0, 1]}$. Let $u \in L ^1 (\Omega)$ with weak partial derivatives bounded in inhomogeneous spaces 
    \begin{align*}
        \pt u &\in \Lt 1 \Lz \infty (\Omega) + \Lt q \Lz 1 (\Omega),
        &
        \pz u &\in \Lt p \Lz \infty (\Omega) + \Lt \infty \Lz r (\Omega),
    \end{align*}
    with $p > \frac1\alpha, q > \frac1{1 - \alpha}, r > 1$, then $u \in C (\Omega)$ is continuous with oscillation bounded by  
    \begin{align*}
        \sup _{\Omega} u - \inf _{\Omega} u = \nor{u} _{\osc (\Omega)} \le C \pth{
            \nor{\pt u} _{\Lt1 \Lz \infty + \Lt q \Lz 1}
            +
            \nor{\pz u} _{\Lt p \Lz \infty + \Lt{\infty} \Lz r}
        }
    \end{align*}
    where $C = C (p, q, r)$ depends on $p, q, r$. 
\end{lemma}

\begin{figure}[htbp]
    \centering
    \includegraphics{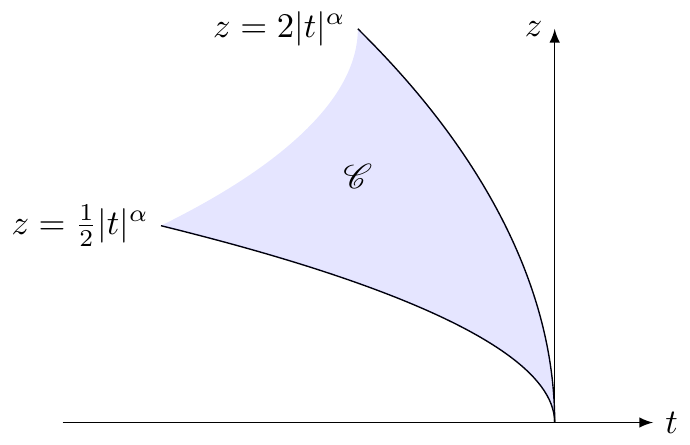}
    \caption{Inhomogeneous Sobolev Embedding}
    \label{fig:sobolev}
\end{figure}

\begin{proof}
    Up to cutoff and mollification, we may assume $u \in C ^\infty ((-\infty, 0] \times [0, \infty))$ with compact support in $2 \Omega = (-2, 0] \times [0, 2)$. Up to translation, we show $u (0, 0)$ is bounded. By the fundamental theorem of calculus, for any $\lambda > 0$, we have
    \begin{align*}
        0 &= u (0, 0) + \int _0 ^\infty \pfr{}{s} u (-s, \lambda s ^\alpha) \d s. 
    \end{align*}
    Taking average for $\lambda \in \pth{\half, 2}$ yields
    \begin{align*}
        \abs{u (0, 0)} &\le \int _{\half} ^2 \int _0 ^\infty \abs{\pt u} + \lambda \alpha s ^{\alpha - 1} \abs{\pz u} \d s \d \lambda.
    \end{align*}
    The Jacobian of $(t, z) = (s, \lambda s ^\alpha)$ is 
    \begin{align*}
        \frac{D (t, z)}{D (s, \lambda)} = \det \begin{bmatrix}
            -1 & 0 \\
            \lambda \alpha s ^{\alpha - 1} & s ^\alpha
        \end{bmatrix} = s ^\alpha = \abs t ^\alpha \sim z,
    \end{align*}
    thus we can bound $u (0, 0)$ via a change of variable by
    \begin{align*}
        \abs{u (0, 0)} 
        &\le 
        \int _{\mathscr C} \pth{\abs{\pt u} + \alpha z \abs t \inv \abs{\pz u}} \abs t ^{-\alpha} \d z \dt 
        \\
        &= 
        \int _{\mathscr C} \abs t ^{-\alpha} \abs{\pt u} + \alpha \lambda \abs t \inv \abs{\pz u} \d z \dt. 
    \end{align*}
    where $\mathscr C$ is the region illustrated in Figure \ref{fig:sobolev}.

    Now we compute inhomogeneous norms of $\abs t \inv$ and $\abs t ^{-\alpha}$ in $\mathscr C$:
    \begin{align*}
        \int _{\half \abs t ^\alpha} ^{2 \abs t ^\alpha} \abs t ^{-\alpha} \d z &= \frac32 \in \Lt\infty (-2, 0), \\
        \nor{\abs t ^{-\alpha}} _{\Lz \infty (\half \abs t ^\alpha, 2 \abs t ^\alpha)} &= \abs t ^{-\alpha} \in \Lt{q'} (-2, 0), \\
        \int _{\half \abs t ^\alpha} ^{2\abs t ^\alpha} \abs t \inv \d z &= \frac32 \abs t ^{\alpha - 1} \in \Lt{p'} (-2, 0), \\
        \nor{1/t} _{\Lz{r'}(\half \abs t ^\alpha, 2 \abs t ^\alpha)} &= \abs t \inv \pth{\frac32 \abs t ^\alpha} ^\frac1{r'} \lesssim \abs t ^{\frac\alpha{r'} - 1} \in \Lt1 (-2, 0).
    \end{align*}
    Here $p' < \frac1\alpha, q' < \frac1{1 - \alpha}, r' < \infty$ are the H\"older conjugate of $p, q, r$ respectively. In conclusion, $\abs t \inv$ and $\abs t ^{-\alpha}$ are bounded in spaces
    \begin{align*}
        \abs t ^{-\alpha} &\in \Lt \infty \Lz 1 \cap \Lt {q'} \Lz \infty, &
        \abs t \inv &\in \Lt {p'} \Lz 1 \cap \Lt 1 \Lz {r'},
    \end{align*}
    which completes the proof of this lemma by H\"older inequality.
\end{proof}

\subsection{Parabolic Maximal Function}

Let us introduce the following notion of maximal function adapted to the parabolic scaling. 

\begin{definition}[Parabolic Maximal Function]
    \label{def:maximal}
    For $f \in L _{\loc} ^1 (\R \times \Rd)$, we define the parabolic maximal function by taking the greatest mean values
    \[
        \mm f (t, x) := \sup _{r > 0} \fint _{t - r ^2} ^{t + r ^2} \fint _{B _r (x)} \vert f (s, y)\vert \d y \d s.
    \]
    For $f \in L ^1 ((0, T) \times \Omega)$ where $\Omega \subset \Rd$ is a bounded set, we can define $\mm f$ by applying the previous definition on the zero extension of $f$ in $\R \times \Rd$. 
\end{definition}

Recall the classical weak type $(1, 1)$ bound on the maximal function $\mm$: 
\[
    \nor {\mm f} _{L ^{1, \infty}} \le C _d \nor f _{L ^1}.
\]

\subsection{Lipschitz Decay of 1D Heat Equation}

We end this section by reminding the readers that solutions to the 1D heat equation have a decay rate of $t ^{-\frac34}$ in the Lipschitz norm. 
It will be useful to control the Prandtl layer in a small initial time of order $O (\nu ^3)$.
This result is very elementary. We give the proof for the sake of completeness.
\begin{lemma}
    For $z > 0$ we have
    \[ 
        \sum _{n = 1} ^\infty n ^2 e ^{-n ^2 z} < z ^{-\frac32}.
    \] 
\end{lemma}

\begin{proof}
    We can approximate this infinite series by
    \begin{align*}
        \sum _{n = 1} ^\infty n ^2 e ^{-n ^2 z} &= z^{-\frac32} \sum _{n = 1} ^\infty (\sqrt{z}n) ^2 e ^{-(\sqrt zn) ^2} \sqrt z \\
        &= z ^{-\frac32} \pth{
            \int _0 ^\infty x ^2 e ^{-x ^2} \dx + O (\sqrt z)
        } \\
        &= \frac{\sqrt\pi}4 z ^{-\frac32} + O (z ^{-1}),
    \end{align*}
    when $z \to 0$ is small, and 
    \begin{align*}
        \sum _{n = 1} ^\infty n ^2 e ^{-n ^2 z} 
        &\le
        \sum _{n = 1} ^\infty n ^2 e ^{-n z}  
        = \dfr{^2}{z^2} \pth{
            \sum _{n = 1} ^\infty e ^{-n z}
        }  
        = \dfr{^2}{z^2} \pthf1{
           e^z - 1
        }
        = \frac{(e^z + 1) e^z}{(e ^z - 1) ^3} \approx e ^{-z}
    \end{align*}
    when $z \to \infty$ is large. This proves that the left hand side is bounded by $C z ^{-\frac32}$ for some constant $C$, which can be easily determined by carefully examine the estimates.
\end{proof}

Using this lemma, we can compute the decay rate.

\begin{lemma}
    \label{lem:heat}
    Let $\nu > 0$, $H > 0$, and suppose $v (t, x _d)$ solves the following 1D heat equation in $[0, H]$:
    \begin{align*}
        \begin{cases}
            \pt v = \nu v _{xx} & \inn (0, \infty) \times (0, H) \\
            v = 0 & \onn (0, \infty) \times \{0, H\} \\
            v = v _0 & \att t = 0
        \end{cases}
    \end{align*}
    with $v _0 \in L ^2 (0,H)$. Then
    \begin{align*}
        \nmL \infty{\grad v (t)} \le \half (\nu t) ^{-\frac34} \nmL 2{v _0}.
    \end{align*}
\end{lemma}

\begin{proof}
    We can write the solutions explicitly in terms of Fourier series. We expand $v _0$ by sine series as 
    \begin{align*}
        v _0 (x) = \sum _{n = 1} ^\infty b _n \sin \pthf{n \pi x}H,
    \end{align*}
    with 
    \begin{align*}
        \sum _{n = 1} ^\infty b _n ^2 = \frac2H \nmL 2{v _0} ^2 < \infty.
    \end{align*}
    The solution can be explicitly written as 
    \begin{align*}
        v (t, x) = \sum _{n = 1} ^\infty b _n \sin \pthf{n \pi x}H e ^{- \nu \frac{n ^2 \pi ^2}{H ^2} t},
    \end{align*}
    so the derivative is bounded by
    \begin{align*}
        \abs{\partial _x v (t, x)} & \le \abs{
            \sum _{n = 1} ^\infty b _n \cos \pthf{n \pi x}H \pthf{n \pi}H e ^{- \nu \frac{n ^2 \pi ^2}{H ^2} t}
        } \\
        & \le \pth{
            \sum _{n = 1} ^\infty b _n ^2
        } ^\half \pth{
            \sum _{n = 1} ^\infty \pthf{n \pi}H ^2 e ^{- 2 \nu \frac{n ^2 \pi ^2}{H ^2} t},
        } ^\half \\
        &\le \pth{
            \frac2H 
        \nmL 2{v _0} ^2
        } ^\half 
        \pthf\pi H
        \pthf{2 \nu \pi ^2 t}{H ^2} ^{-\frac34} \\
        &\le \frac12 (\nu t) ^{-\frac34} \nmL 2{v _0}
    \end{align*}
    using the previous lemma.
\end{proof}

\section{Boundary Regularity for the Navier-Stokes Equation}
\label{sec:boundary}

The goal of this section is to prove the boundary regularity for the Navier-Stokes equation with unit viscosity constant: Theorem \ref{thm:boundary-regularity}. This relies on the following local estimate. 

\begin{proposition}
    \label{thm:local}
    Suppose $(u,P)$ is a weak solution to the Navier-Stokes equation \eqref{eqn:nse} with forcing term $f \in L ^1 (-4, 0; L ^2 (\Bp2))$, such that $u \in L ^\infty(-4, 0; L ^2(\Bp2))$, $\nabla u \in L^2(\Qp2)$, and in distribution they satisfy
    \begin{align*}
        \begin{cases}
            \pt u + u \cdot \grad u + \grad P = \La u + f & \inn \Qp2 \\
            \div u = 0 & \inn \Qp2 \\
            u = 0 & \onn \Qb2
        \end{cases}.
    \end{align*}
    If we denote 
    \begin{align*}
        c _0 := \int _{-4} ^0 \nor{\grad u (t)} _{L ^2 (\Bp2)} ^2 + \nor{f} _{L ^2 (\Bp2)} \dt,
    \end{align*}
    then we can bound the average-in-time vorticity on the boundary by
    \[
        \int _{\Bb1} \abs{
            \int _{-1} ^0 \omega (t, x', 0) \dt
        } \dx' \le C (c _0 + c _0 ^\frac1{2}).
    \]
\end{proposition}

\begin{proof}
    For $t \in (-3, 0)$, we define 
    \[
        U (t, x) = \int _{t - 1} ^t u (s, x) \d s.
    \]
    As explained in the introduction, this is needed to tame the time oscillation of the local pressure, which comes from $\pt u$. This allows us to apply the local Stokes estimate at the boundary. 
    Denote $\rho (t) = \ind{[0,1]} (t)$, then $U = u *_t \rho$, where $* _t$ stands for convolution in $t$ variable only. If we denote $Q = P *_t \rho$, and $F = (f - u \cdot \grad u) *_t \rho$, then $U$ satisfies the following system:
    \begin{align*}
        \begin{cases}
            \pt U + \grad Q = \La U + F & \inn (-3, 0) \times \Bp2 \\
            U = 0 & \onn \set{x _d = 0}
        \end{cases}.
    \end{align*}
    The proof of this theorem can be divided into three steps: the first two estimate terms in this system, and the last step uses the Stokes estimate and the Sobolev embedding.
    
    \paragraph*{\bf Step 1. Estimates on $u, U, \pt U, \La U$.}

    We have via Sobolev embedding and using that $u = 0$ on $\Qb2$ that 
    \begin{align}
        \label{eqn:nor-u-2}
        \nor{
            u
        } _{\Lt {2} \Lx {6}\pth{\Qp2}} \le C c _0 ^{\frac1{2}}
    \end{align}
    for both dimension 2 and 3.
    Since $\pt U (t, x) = u (t, x) - u (t - 1, x)$, we have 
    \begin{align*}
        \nor{\pt U} _{\Lt {2} \Lx {6} ((-3, 0) \times \Bp2)} \le C c _0 ^{\frac1{2}},
    \end{align*}
    On the other hand, the Laplacian of $U$ is bounded by
    \begin{align*}
        \nor {\La U} _{\Lt \infty H _x ^{-1} ((-3, 0) \times \Bp2)} 
        \le 
        C \nor {\La u} _{\Lt 2 H _x ^{-1} (\Qp2)} 
        \le 
        C \nor {\grad u} _{L ^2 (\Qp2)} 
        \le 
        C c _0 ^\frac12.
    \end{align*}

    \paragraph*{\bf Step 2. Estimates on $F$ and $Q$.}

    Applying H\"older's inequality, by \eqref{eqn:nor-u-2} we have
    \begin{align*}
        \nor{u \cdot \grad u} _{\Lt 1 \Lx {
        \frac32} (\Qp2)} \le C c _0.
    \end{align*}
    Also by \eqref{eqn:nor-u-2} we have by embedding that
    \begin{align*}
        \nor{\div(u \tensor u)} _{\Lt 1 W _x ^{-1, 3}(\Qp2)} \le C c _0.
    \end{align*}
    for both dimension 2 and 3.
    By convolution, we bound $F$ by 
    \begin{align*}
        \nor{F} _{\Lt \infty \Lx {\frac32} ((-3, 0) \times \Bp2)}, \nor{F} _{\Lt \infty W _x ^{-1, 3} ((-3, 0) \times \Bp2)} \le C c _0.
    \end{align*}
    Next we estimate $Q$. Using $\grad Q = \La U + F - \pt U$ we have 
    \begin{align*}
        \nor {\grad Q} _{\Lt {2} H _x ^{-1}} \le C c _0 + 
        C c _0 +
        C c _0 ^\frac12 \le C (c _0 + c _0 ^\frac1{2}).
    \end{align*}
    Without loss of generality we assume that the average of $Q$ is zero at every $t$. Then by Ne\v cas theorem (see \cite{Seregin2014}, Section {1.4}), 
    \begin{align*}
        \nor {Q} _{L ^2 _{t, x}} \le C (c _0 + c _0 ^\frac1{2}).
    \end{align*}

    \paragraph*{\bf Step 3. Stokes estimates and Trace theorem.}
    
    By Corollary \ref{cor:mixed}, we can split $U = U _1 + U _2$, where for any $p < \infty$, we have
    \begin{align*}
        \nor{
            \abs{\pt U _1} + \abs{\grad ^2 U _1}
        } _{\Lt p \Lx {\frac32} (\Qp 1)} 
        + 
        \nor{
            \abs{\pt U _2} + \abs{\grad ^2 U _2}
        } _{\Lt 2 \Lx p (\Qp 1)}
        \le 
        C (c _0 + c _0 ^\frac12).
    \end{align*}
    Denote $\Omega (t, x _d) := \int _{\Bb 1} \abs{\grad U (t, x', x _d)} \d x'$, then $\partial _{x _d} \Omega$ is bounded in 
    \begin{align*}
        \partial _{x _d} \Omega \in \Lt 2 L _{x _d} ^p + \Lt p L _{x _d} ^{\frac32} ((-1, 0) \times (0, 1)).
    \end{align*}
    for any $p < \infty$. Note that 
    \begin{align*}
        \pt \Omega = \int \abs{\grad u} \d x' \in L ^2 _{t, x _d} ((-1, 0) \times (0, 1)).
    \end{align*}
    Since by interpolation, $\Lt 1 L _{x _d} ^\infty \cap \Lt \infty L _{x _d} ^1 \subset L ^2 _{t, x _d}$, by duality $\pt \Omega$ is bounded in $L ^2 _{t, x _d} \subset \Lt 1 L _{x _d} ^\infty + \Lt \infty L _{x _d} ^1$. Similarly, $\partial _{x _d} \Omega$ is bounded in 
    \begin{align*}
        \partial _{x _d} \Omega \in \Lt 2 L _{x _d} ^p + \Lt p L _{x _d} ^{\frac32} ((-1, 0) \times (0, 1)) \subset \Lt r L _{x _d} ^\infty + \Lt \infty L _{x _d} ^r ((-1, 0) \times (0, 1))
    \end{align*}
    for any $p > 3$ and $r > 1$ sufficiently small. 
    Now we can use Lemma \ref{lem:sob} to show $\Omega$ is continuous up to the boundary with oscillation bounded by
    \begin{align*}
        \nor{\Omega} _{\osc ((-1, 0) \times (0, 1))} \le C (c _0 + c _0 ^\frac12).
    \end{align*}
    Since the average of $\Omega$ is also bounded as 
    \begin{align*}
        \int \Omega \d {x _d} \dt = \int _{\Qp1} \abs{\grad u} \dx \dt \le C c _0 ^\half,
    \end{align*}
    we have $\Omega$ is bounded in $L ^\infty$, in particular
    \begin{align*}
        \int _{\Bb1} \abs{
            \int _{-1} ^0 \grad u (t, x', 0) \dt
        } \dx' = \Omega (0, 0) \le C_0 (c _0 + c _0 ^\half).
    \end{align*}
    This concludes the proof of this proposition.
\end{proof}

The proof of Theorem  \ref{thm:boundary-regularity} relies on a domain decomposition inspired by the Calder\'on–Zygmund decomposition introduced for the study of singular integrals (see \cite{Stein1993}).
We first define the parabolic dyadic decomposition.

\begin{definition}[Parabolic Dyadic Decomposition]
    \label{def:parabolic-dyadic-decomposition}
    Let $L > 0$, and let $\Omega$ be a periodic channel of period $W$ and height $H$. We define the parabolic dyadic decomposition of $(0, L) \times \Omega$ as below. Denote
    \begin{align}
        \label{eqn:short-in-time}
        R _0 = \min \set{
            \sqrt L, \frac W2, \frac H2
        }.
    \end{align}
    Then we can find positive integer $k _L, k _W, k _H$, such that 
    \begin{align*}
        L = 4 ^{k _L} L _0, 
        \qquad 
        W = 2 \cdot 2 ^{k _W} W _0, 
        \qquad 
        H = 2 \cdot 2 ^{k _H} H _0,
    \end{align*}
    where $L _0, W _0, H _0$ satisfy 
    \begin{align*}
        R _0 \le \sqrt{L _0}, W _0, H _0 \le 2 R _0.
    \end{align*}
    First, we evenly divide $(0, L) \times \Omega$ into $4 ^{k _L} \cdot 2 ^{k _W + 1} \cdot 2 ^{k _H + 1}$ cubes of length $L _0$, width $W _0$ and height $H _0$, and denote $\mQ 0$ to be this set of cubes. For each $Q \in \mQ 0$, we can divide $Q$ into $4 \times 2 ^d$ subcubes with length $L _0 / 4$, width $W _0 / 2$, and height $H _0 / 2$. This set is denoted by $\mQ 1$. For each cube in $\mQ 1$, we can continue to dissect it into $4 \times 2 ^d$ smaller cubes with a quarter the length, half the width, and half the height. We denote the resulted family by $\mQ 2$. We proceed indefinitely and define $\mathcal Q = \cup _{k \in \mathbb N} \mQ k$ to be the parabolic dyadic decomposition of $(0, L) \times \Omega$. 
\end{definition}

\begin{proof}[Proof of Theorem \ref{thm:boundary-regularity}.]
    The partition of $(0, T) \times \Omega$ is constructed as follows. Among the parabolic dyadic decomposition of $(0, T) \times \Omega$, we first select a family of disjoint cubes, denoted by $\mQo$, according to the following rule:
    \begin{enumerate}[\bfseries a)]
        \item \label{cond1} For any integer $k \ge 1$, in $\set{4 ^{-k} L _0 \le t \le 4 ^{-k+1} L _0}$, we pick every parabolic cube in $\mQ k$, which are cubes of size $4 ^{-k} L _0 \times 2 ^{-k} W _0 \times 2 ^{-k} H _0$.
        \item \label{cond2} In $\set{t \ge L _0}$, we pick every parabolic cube in $\mQ 0$.
    \end{enumerate}
    The selection of these cubes ensures enough gap from the initial time $t = 0$, which allows the local parabolic regularization to apply around these cubes.
    \begin{figure}[bp]
        \centering
        \includegraphics{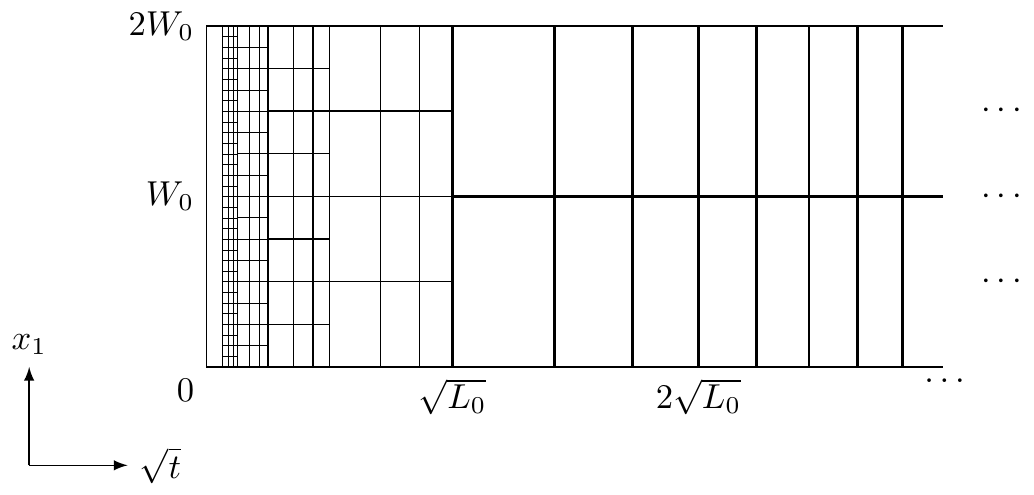}
        
        \caption{Initial Partition $\mQo$ of a Long Channel $(0, L) \times \Omega$}
        \label{fig:initial-decomposition-2}
    \end{figure}
    \begin{figure}[tbp]
        \centering
        \includegraphics{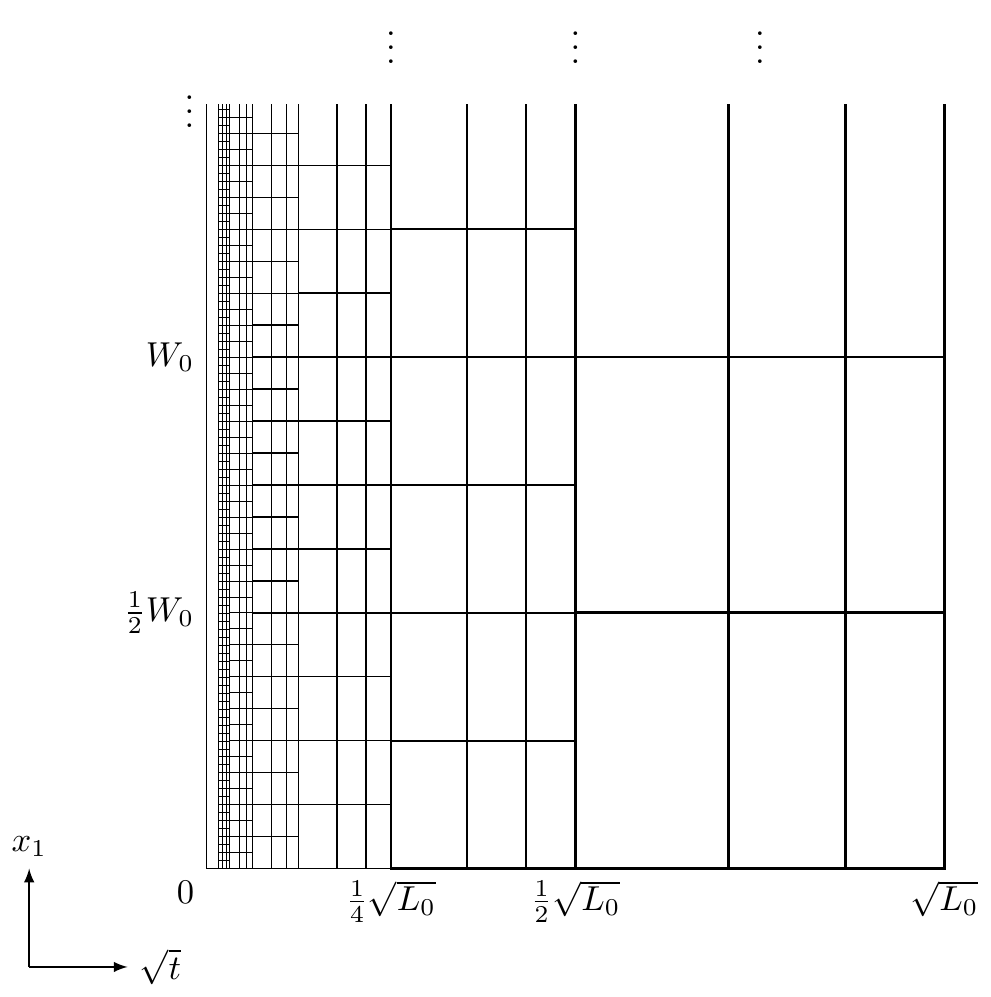}
        \caption{Initial Partition $\mQo$ of a Wide Channel $(0, L) \times \Omega$}
        \label{fig:initial-decomposition}
    \end{figure}
    As shown in Figure \ref{fig:initial-decomposition-2} and Figure \ref{fig:initial-decomposition}, they form a partition of $(0, T) \times \Omega$. Figure \ref{fig:initial-decomposition-2} corresponds to when $R _0 = \mins{\frac W2, \frac H2} < \sqrt {L _0}$, and figure \ref{fig:initial-decomposition} corresponds to when $R _0 = \sqrt {L _0} = \sqrt T$, in which case \textbf{\ref{cond2})} does not happen.

    We are interested in cubes that touch the boundary, i.e., having zero distance from $\partial \Omega$. We call these cubes the ``boundary cubes''. 
    Given a boundary cube $Q \in \mQ k$ that meets the boundary $\{x _d = 0\}$, we denote its length as $l = 4 ^{-k} L _0$, width as $w = 2 ^{-k} W _0$, and height as $h = 2 ^{-k} H _0$. Thus for some $(t, x', 0) \in (0, T) \times \pOmega$, $Q$ can be expressed as
    \begin{align*}
        Q &= (t - l, t) \times \bar B _{w / 2} (x') \times (0, h), \qquad \bar B _{w / 2} (x') = \set{y': \nor{x' - y'} _{\ell ^\infty} < w / 2}
    \end{align*}
    Let us denote 
    \[
        2 Q = (t - 2 l, t) \times \bar B _w (x') \times (0, 2 h).
    \]
    Similar definition applies to boundary cubes that touch $\set{x _d = H}$. A boundary cube $Q \in \mQ k$ is said to be suitable if it satisfies 
    \begin{align}
        \label{eqn:suitable}
        \tag{S}
        \fint _{2 Q} \abs{\grad u} ^2 \dx \dt \le c _0 (2 ^{-k} R _0) ^{-4}
    \end{align}
    for some $c _0$ to be determined.

    Starting from $\mQo$, we decompose the boundary cubes based on the following rules. 
    For each boundary cube in the initial partition $\mQo$ that is not suitable, we dyadically dissect it into $4 \times 2 ^d$ smaller parabolic cubes. For each smaller boundary cube, we continue to dissect it until the suitability condition \eqref{eqn:suitable} is satisfied. This process will finish in finitely many steps almost everywhere because $\grad u$ is bounded in $L ^2$ for any Leray-Hopf solutions, so all sufficiently small cubes are suitable.

    The final partition will contain a subcollection of dyadic boundary cubes $\set{Q ^i} _{i \in \Lambda} \subset \mathcal Q$ that are suitable, mutually disjoint, and verify $\overline{(0, T) \times \pOmega} = \overline {\bigcup _i \Qb i}$. For each boundary cube $Q ^i \in \mQ k$ centered at $(t \pp i, x \pp i)$, we denote its length as $l _i = 4 ^{-k} L _0$, width as $w _i = 2 ^{-k} W _0$, and height as $h _i = 2 ^{-k} H _0$. Thus $Q ^i$ can be expressed as
    \begin{align*}
        Q ^i &= (t \pp i - l _i, t \pp i) \times \Bbi \times (0, h _i), 
        & 
        \Bbi &= \bar B _{w _i / 2} (x \pp i).
    \end{align*}
    It is easy to see from our construction that $2 Q ^i \subset (0, T) \times \Omega$. Denote $r _i = 2 ^{-k} R _0$, then from Definition \ref{def:parabolic-dyadic-decomposition} we have 
    \[
        r _i \le \sqrt{l _i}, w _i, h _i \le 2 r _i.
    \]
    Suitability \eqref{eqn:suitable} of $Q ^i$ implies 
    \begin{align*}
        \fint _{2 Q ^i} \abs{\grad u} ^2 \dx \dt \le c _0 r _i ^{-4}.
    \end{align*}
    Using the canonical scaling of the Navier-Stokes equation $u _r (t, x) := r u (r ^2 t, r x)$, Proposition \ref{thm:local} implies that 
    \[
        \tomega \vert_{\Qbi} = \fint _{\Bbi} \abs{\fint _{t \pp i - l _i} ^{t \pp i} \omega (t, x', 0) \dx'} \dt \le C (c _0 + c _0 ^\frac12) r _i ^{-2} =: c _1 r _i ^{-2}.
    \]
    We can use this Proposition because $Q ^i$ is comparable to a parabolic cube. 

    Now we separate three cases:
    \begin{enumerate}
        \item If $Q ^i \in \mQo \cap \mQ k$ with $k \ge 1$, then by condition \textbf{\ref{cond1})}, any $(t, x) \in Q ^i$ satisfies $t < 4 l _i \le 16 r _i ^2$, thus in $\Qbi$ we have 
        \begin{align*}
            \tomega \le \frac {16 c _1}t.
        \end{align*}
        We can select $c _0$ small enough such that $16 c _1 = 1$.

        \item If $Q ^i \in \mQo \cap \mQ 0$, then by condition \textbf{\ref{cond2})}, any $(t, x) \in Q ^i$ satisfies $L _0 = l _i < t < T$, $r _i = R _0$, thus in $\Qbi$ we have 
        \begin{align*}
            \tomega \le c _1 R _0 ^{-2} = \frac1{16} R _0 ^{-2}, 
        \end{align*}
        Note that this case only happen when $T > L _0 \ge R _0 ^2$, so in fact we know $R _0 = \min \set{W, H} / 2$, thus $\tomega \le \mins{W, H} ^{-2}$.
        
        \item If $Q ^i \notin \mQo$ is not one of the initial cubes in the grid, then its antecedent cube $\tilde Q ^i$ is also a boundary cube and is not suitable, so 
        \begin{align*}
            \fint _{2 \tilde Q ^i} \abs{\grad u} ^2 \dx \dt > c _0 (2 r _i) ^{-4},
        \end{align*}
        By the definition of the maximal function $\mm$ (recall Definition \ref{def:maximal}), this implies
        \begin{align*}
            \min _{Q ^i} \mm (\abs{\grad u} ^2) \ge c _2 r _i ^{-4}.
        \end{align*}
        for some $c _2$ comparable to $c _0$.
    \end{enumerate}

    Combining these three cases, for any $\rstar = 2 ^{l} R _0$ with $l \in \mathbb Z$, we have 
    \begin{align*}
        &
        \hspace{-3em}
        \set{
            (t, x') \in (0, T) \times \partial \Omega: \tomega > \maxs{
                c _1 \rstar ^{-2}, 
                t \inv, 
                W ^{-2}, 
                H ^{-2}
            }
        } 
        \\
        & 
        \subset 
        \bigcup _{i} \set{
            \Qbi : r _i < \rstar
        }
        \subset 
        \bigcup _{i} \bigcup _{k = 1} ^\infty \set{
            \Qbi : r _i = 2 ^{-k} \rstar
        }.
    \end{align*}
    Therefore the measure of the upper level set is controlled by the total measure of these suitable boundary cubes, that is
    \begin{align*}
        \abset{
            \tomega > \maxs{
                c _1 \rstar ^{-2}, 
                t \inv, 
                W ^{-2}, 
                H ^{-2}
            }
        } 
        &\le 
        \sum _{k = 1} ^\infty \sum _{r _i = 2 ^{-k} \rstar} \abs{\Qbi} 
        \\
        &\le 
        \sum _{k = 1} ^\infty \frac{2 ^k}{\rstar} \sum _{r _i = 2 ^{-k} \rstar} \abs{Q ^i} .
    \end{align*}
    Note that 
    \begin{align*}
        \bigcup _{i} \set{
            Q ^i : r _i = 2 ^{-k} \rstar
        } \subset \set{
            \mm (\abs{\grad u} ^2) \ge c _2 (2 ^{-k} \rstar) ^{-4}
        },
    \end{align*}
    which implies that
    \begin{align*}
        &\hspace{-2em}
        \abset{
            \tomega > \maxs{
                c _1 \rstar ^{-2}, 
                t \inv, 
                W ^{-2}, 
                H ^{-2}
            }
        } 
        \\
        &\le 
        \sum _{k = 1} ^\infty \frac{2 ^k}{\rstar} \abset{
            \mm (\abs{\grad u} ^2) \ge c _2 (2 ^{-k} \rstar) ^{-4}
        } 
        \\
        &\lesssim 
        \sum _{k = 1} ^\infty \frac{2 ^k}{\rstar} \nor{
            \mm (\abs{\grad u} ^2)
        } _{L _{\loc} ^{1,\infty} ((0, T) \times \Omega)} (2 ^{-k} \rstar) ^{4} 
        \\
        &\lesssim 
        \nmLWT1{
            \abs{\grad u} ^2
        } \rstar ^{3}.
    \end{align*}
    By the definition of Lorentz space, this shows
    \begin{align*}
        \nmLpWT{\frac32, \infty}{\tomega \indtomega} ^{\frac32} \lesssim \nmLWT2{\grad u} ^2.
    \end{align*}
    This completes the proof of the theorem.
\end{proof}

\section{Proof of the Main Result}
\label{sec:main}

This section is dedicated to the proof of Theorem \ref{thm:main}.
Theorem \ref{thm:boundary-regularity} provides a control on the large part of $\tomega$, but it leaves a remainder in the region $\tomega < \frac1t$, whose integral has a logarithmic singularity at $t = 0$. To avoid this singularity, we should apply Theorem \ref{thm:boundary-regularity} only away from $t = 0$, and near $t = 0$ we should adopt a different strategy. 

Let $\unup$ be a shear solution to \eqref{eqn:nse-nu} with initial value $\unup \at{t = 0} = \ub$ (the pressure term is 0). Then $\unup$ can be written as
\begin{align*}
    \unup (t, x) = 
    \begin{cases}
        \Unup (t, x _2) e _1 & d = 2 \\
        {\Unup} _1 (t, x _3) e _1 + {\Unup} _2 (t, x _3) e _2 & d = 3
    \end{cases},
\end{align*}
where $\Unup$ solves the Prandtl layer equation, 
\begin{align}
    \label{eqn:prandtl-nu}
    \tag{$\text{Pr} _\nu$}
    \begin{cases}
        \pt \Unup = \nu \partial _{x _d x _d} \Unu & \inn (0, T) \times (0, H) \\
        \Unup = 0 & \onn (0, T) \times \{0, H\} \\
        \Unup = \bar U & \att t = 0
    \end{cases}.
\end{align}

We choose a small positive number $T _\nu < T$ to be determined later, and separate the evolution into two parts: in a short period $(0, T _\nu)$, we compare $\unu$ and $\ub$ with the Prandtl layer $\unup$, while in the remaining time $(T _\nu, T)$, we compare $\unu$ and $\ub$ using the boundary vorticity.

Before we proceed, let us remark on a few useful computations and estimates that will be used repeatedly in this section. If $v, w$ are two divergence-free vector fields in $(0, T) \times \Omega$ satisfying the no-slip boundary condition $v = 0$ and the no-flux boundary condition $w \cdot n = 0$ on $\pOmega$ respectively, then we have the following three estimates:
\begin{align}
    \label{eqn:tools-from}
    (v - w, v \cdot \grad v - w \cdot \grad w) 
    &= (v - w, v \cdot \grad (v - w)) + (v - w, (v - w) \cdot \grad w) \\
    \notag
    &\le \nmL\infty{\grad w} \nmL2{v - w} ^2, \\
    (v - w, \grad P) 
    &= \int _{\pOmega} P (v - w) \cdot n \d S = 0, \\
    \label{eqn:tools-to}
    (v - w, \La v) &= - \nmLW2{\grad v} ^2 + (\grad w, \grad v) - \int _{\pOmega} w \cdot \partial _n v \d S \\
    \notag
    &\le - \half \nmL2{\grad v} ^2 + \half \nmL2{\grad w} - \int _{\pOmega} J[w] \cdot \curl v \d S .
\end{align}
Here $J[w]$ is a rotation of $w$ and $\curl v$ is the vorticity of $v$ defined by 
\begin{align*}
    J[w] := \begin{cases}
        n ^\perp \cdot w & d = 2 \\
        n \times w & d = 3
    \end{cases}, 
    \qquad 
    \curl v := \begin{cases}
        \grad ^\perp \cdot v & d = 2 \\
        \grad \times v & d = 3
    \end{cases},
\end{align*}
where $n ^\perp$ is the rotation of the normal vector counterclockwise by a right angle, and $\grad ^\perp = (-\partial _{x _2}, \partial _{x _1})$. Moreover, note that $w \cdot \grad w = 0$ in \eqref{eqn:tools-from} when $w$ is a shear flow.

\subsection{Prandtl Timespan} 

To compute the evolution of $\unu - \unup$, first we subtract their equations and obtain
\begin{align*}
    \pt (\unu - \unup) + \unu \cdot \grad \unu + \grad \Pnu = \nu \La (\unu - \unup).
\end{align*}
The evolution of $\unu - \unup$ can be computed using \eqref{eqn:tools-from}--\eqref{eqn:tools-to} as
\begin{align*}
    \half \ddt \nmL2{\unu - \unup} ^2 + \nu \nmL2{\grad (\unu - \unup)} ^2 &\le -(\unu - \unup, \unu \cdot \grad \unu) \\
    &\le \nmL\infty{\grad \unup} \nmL2{\unu - \unup} ^2 .
\end{align*}
By Lemma \ref{lem:heat}, the Lipschitz norm of the Prandtl layer at time $t$ is
\begin{align*}
    \nmL\infty{\grad \unup}(t) = \nmL\infty{\grad \Unup} \le \half (\nu t) ^{-\frac34} \pthf{E}{\vert\pOmega\vert} ^\half .
\end{align*}
Integrating in time, we have
\begin{align}
    \label{eqn:L1inf}
    2 \nor {\grad \unup} _{L ^1 (0,T_\nu; L ^\infty (\Omega))} \le \int _0 ^{T _\nu} (\nu t) ^{-\frac34} \pthf{E}{\vert\pOmega\vert} ^\half \dt \le \log 2
\end{align}
if we choose $T _\nu$ small enough such that 
\begin{align}
    \label{eqn:def-T*}
    T _\nu \le T _* := \pthf{\log 2}{4} ^4 E ^{-2} \vert\pOmega\vert ^{2} \nu ^3.
\end{align}
By Gr\"onwall's inequality, we have for any $0 < t < T _\nu$, 
\begin{align}
    \label{eqn:0-Tnu-1}
    \half\nmLW2{\unu - \unup} ^2 (t) + \nu  \nmLWt2{\grad (\unu - \unup)} ^2 
    \le \nmLW2{\unu - \ub} ^2 (0).
\end{align}

The evolution of $\unup - \ub$ can be computed using \eqref{eqn:tools-to} as
\begin{align*}
    \half \ddt \nmLW2{\unup - \ub} ^2 &= (\unup - \ub, \pt \unup) 
    = (\unup - \ub, \nu \La \unup) \\
    &\le  -\frac\nu2 \nmL2{\grad \unup} ^2 + \frac\nu2 \nmL2{\grad \ub} ^2 - \nu \int _{\pOmega} \ub \cdot \partial _n \unup \dx'
\end{align*}
where $\nmL2{\grad \ub} ^2 \le G ^2 \vert\Omega\vert$ and
\begin{align*}
    \abs{
        \int _{\pOmega} \ub \cdot \partial _n \unup \dx'
    } \le \nmLpW\infty{\grad \unup} \nmLpW\infty{\ub} \vert\pOmega\vert.
\end{align*}
Integration in time gives for any $0 < t < T _\nu$, we have 
\begin{align*}
    & \half \nmLW2{\unup - \ub} ^2 (t) + \frac\nu2 \nmLWt2 {\grad \unup} ^2 \\
    & \qquad \le \frac\nu2 G ^2 \vert\Omega\vert t + A \nu \vert\pOmega\vert \nor {\grad \unup} _{L ^1 (0, T_\nu; L ^\infty (\Omega))}  \\
    & \qquad \le \frac\nu2 G ^2 \vert\Omega\vert t + \half A ^2 \abs{\Omega} \Re \inv
\end{align*}
where the last inequality used \eqref{eqn:L1inf}.

Combined with \eqref{eqn:0-Tnu-1}, we have for any $0 < t \le T _\nu$,
\begin{equation}
    \label{eqn:0-Tnu}
    \begin{aligned}
        &
        \half \nmLW2{\unu - \ub} ^2 (t) + \frac\nu2\nmLWt2{\grad \unu} ^2 
        \\
        & \qquad \le 2 \nmLW2{\unu - \ub} ^2 (0) + \nu G ^2 \vert\Omega\vert t + A ^2 \abs{\Omega} \Re \inv.
    \end{aligned}
\end{equation}

\subsection{Main Timespan}

The evolution of $\unu - \ub$ can be computed using \eqref{eqn:tools-from}--\eqref{eqn:tools-to} as
\begin{align*}
    \half \ddt \nmL2{\unu - \ub} ^2 &= (\unu - \ub, \pt \unu) \\
    &\le -(\unu - \ub, \unu \cdot \grad \unu) - (\unu - \ub, \grad \Pnu) + \nu (\unu - \ub, \La \unu) \\
    &\le \nmL\infty{\grad \ub} \nmL2{\unu - \ub} ^2 - \half \nu \nmL2{\grad \unu} ^2 + \half \nu \nmL2{\grad \ub} ^2 \\
    &\qquad - \int _{\pOmega} J[\ub] \cdot (\nu \omega ^\nu) \d x'.
\end{align*}
Since $\ub$ is a constant on each connecting component of $\pOmega$, by integrating from $T _\nu$ to $T$, we have 
\begin{align*}
    &\half \nmLW2{\unu - \ub} ^2 (T) + \frac\nu2 \nor{\grad \unu} _{L ^2 ((T _\nu, T) \times \Omega)} ^2 \\
    &\qquad \le \half \nmLW2{\unu - \ub} ^2 (T _\nu) + G \int _{T _\nu} ^T \nmLW2{\unu - \ub} ^2 (t) \dt + \frac\nu2 G ^2 (T - T _\nu) \vert\Omega\vert \\
    &\qquad \qquad + A \pth{
        \abs{ 
            \int _{T _\nu} ^T \intset{x _d = 0} \nu \omega ^\nu \dx' \d t
        } + 
        \abs{ 
            \int _{T _\nu} ^T \intset{x _d = H} \nu \omega ^\nu \dx' \d t
        } 
    }.
\end{align*}
Adding \eqref{eqn:0-Tnu} at $t = T _\nu$, we have for any $T > T _\nu$ that 
\begin{equation}
    \label{eqn:combined}
    \begin{aligned}
        &\half \nmLW2{\unu - \ub} ^2 (T) + \frac\nu2 \nmLWT2{\grad \unu} ^2 \\
        & \le 2 \nmLW2{\unu - \ub} ^2 (0) + G \int _{T _\nu} ^T \nmLW2{\unu - \ub} ^2 (t) \dt 
        + \nu G ^2 T \vert\Omega\vert + A ^2 \abs{\Omega} \Re \inv \\
        &\qquad + A \pth{
            \abs{ 
                \int _{T _\nu} ^T \intset{x _d = 0} \nu \omega ^\nu \dx' \d t
            } + 
            \abs{ 
                \int _{T _\nu} ^T \intset{x _d = H} \nu \omega ^\nu \dx' \d t
            } 
        }.
    \end{aligned}
\end{equation}

\subsection{Proof of Theorem \ref{thm:main}}

We first note that Theorem \ref{thm:main} is only interesting when the initial kinetic energy $\nmL2{\unu (0)}$ and $\nmL2{\ub}$ are comparable.

\begin{lemma}
    Let $\ub \in L ^2 (\Omega)$, and let $\unu$ be a Leray-Hopf solution to \eqref{eqn:nse-nu}, so the energy inequality holds: 
    \begin{align*}
        \half \nmLW2{\unu (T)} ^2 + \nu\nmLWT2{\grad \unu} ^2 \le \half \nmLW2{\unu (0)}.
    \end{align*}
    For any $C' > 1$, there exists $C > 0$ such that if $\nmLW2{\unu (0)} > C \nmLW2{\ub}$ or $\nmLW2{\ub} > C \nmLW2{\unu (0)}$, then
    \begin{align*}
        \nmLW2{\unu (T) - \ub} ^2 + 2 \nu \nmLWT2{\grad \unu} ^2 &\le C' \nmLW2{\unu (0) - \ub} ^2. 
    \end{align*}
\end{lemma}

\begin{proof}
    If $\nmL2{\unu (0)} > C \nmLW2{\ub}$, by energy inequality we can bound
    \begin{align*}
        \nmLW2{\unu (T) - \ub} ^2 &\le \pth{1 + \frac1C} \pth{
            \nmLW2{\unu (T)} ^2
            + C \nmLW2{\ub} ^2
        } \\
        &= \pth{1 + \frac1C} \nmLW2{\unu (0)} ^2 
        - 2 \pth{1 + \frac1C} \nu \nmLWT2{\grad \unu} ^2
        \\
        &\qquad + (C + 1) \nmLW2{\ub} ^2 \\
        &\le \pth{1 + \frac1C} ^2 \pth{
            \nmLW2{\unu (0) - \ub} ^2 + C \nmLW2{\ub} ^2
        } \\
        &\qquad 
        - 2 \nu \nmLWT2{\grad \unu} ^2
        + (C + 1) \nmLW2{\ub} ^2.
    \end{align*}
    Since $\nmL2{\unu (0)} > C \nmL2{\ub}$ implies $\nmL2{\ub} < \frac1{C - 1} \nmL2{\unu (0) - \ub}$, we conclude
    \begin{align*}
        \nmLW2{\unu (T) - \ub} ^2 + 2 \nu \nmLWT2{\grad \unu} ^2 &\le C' \nmLW2{\unu (0) - \ub} ^2
    \end{align*}
    for some $C' \to 1 ^+$ as $C \to \infty$.
    If $\nmL2{\unu (0)} < \frac14 \nmL{2}{\ub}$, then by the energy inequality we can estimate
    \begin{align*}
        \nmLW2{\unu (T) - \ub} ^2 &\le \pth{1 + \frac1C} \pth{
            C \nmLW2{\unu (T)} ^2
            + \nmLW2{\ub} ^2
        } \\
        &\le (1 + C) \nmLW2{\unu (0)} ^2 
        - 2 (1 + C) \nu \nmLWT2{\grad \unu} ^2
        \\
        &\qquad 
        + \pth{1 + \frac1C} \nmLW2{\ub} ^2 \\
        &\le \pth{1 + \frac1C} ^2 \pth{
            \nmLW2{\unu (0) - \ub} ^2 + C \nmLW2{\unu (0)} ^2
        } \\
        &\qquad 
        - 2 \nu \nmLWT2{\grad \unu} ^2
        + (1 + C) \nmLW2{\unu (0)} ^2.
    \end{align*}
    Since $\nmL2{\ub} > C \nmL2{\unu (0)}$ implies $\nmL2{\unu (0)} < \frac1{C - 1} \nmL2{\unu (0) - \ub}$, we again have 
    \begin{align*}
        \nmLW2{\unu (T) - \ub} ^2 + 2 \nu \nmLWT2{\grad \unu} ^2 &\le C' \nmLW2{\unu (0) - \ub} ^2
    \end{align*}
    and the result also follows. 
\end{proof}

Because of this lemma, from here we assume 
\[ 
    \frac E{C} \le \nmLW2{\unu (0)} ^2 \le C E
\]
for some universal constant $C$.
Under this assumption, we see there is a trivial upper bound on layer separation as
\begin{align}
    \label{eqn:trivial-bound}
    \half \nmLW2{\unu (T) - \ub} ^2 + \nu \nmLWT2{\grad \unu} ^2 &\le C E
\end{align}
again using the energy inequality.

Next we study the rescaled boundary vorticity. Since $\unu$ solve \eqref{eqn:nse-nu} in $(0, T) \times \Omega$, its rescale $u (t, x) = \unu (\nu t, \nu x)$ solves \eqref{eqn:nse} in $(0, T / \nu) \times (\Omega / \nu)$. Moreover, 
\[
    \grad u (t, x) = \nu \grad u ^\nu (\nu t, \nu x), \qquad \omega (t, x) = \nu \omega ^\nu (\nu t, \nu x).
\]
Now we apply Theorem \ref{thm:boundary-regularity} on $u$, and we have a rescaled estimate on $\unu$ as 
\begin{align}
    \label{eqn:boundary-vorticity-d=2}
    \nmLpWT{\frac32, \infty}{\nu \tomeganu \indtomeganu} ^\frac32 \le C \nu \nmLWT2{\grad \unu} ^2.
\end{align}

\begin{proof}[Proof of Theorem \ref{thm:main}.]
    We choose $T _\nu = 4 ^{-K} T$ for some integer $K$ such that 
    \begin{align*}
        \frac14 T _* \le T _\nu \le T _*
    \end{align*}
    where $T _*$ is defined in \eqref{eqn:def-T*}. The average of $\omega ^\nu$ in $(T _\nu, T)$ is thus bounded by the average of $\tomeganu$. To estimate the boundary vorticity in \eqref{eqn:combined}, we split it as
    \begin{equation}
        \label{eqn:split}
        \begin{aligned}
            \abs{
                \int _{T _\nu} ^T \intset{x _d = 0} \nu \omega ^\nu \dx' \dt
            } & \le 
                \int_{T _\nu} ^T \intset{x _d = 0} \nu \tomeganu \dx' \dt
             \\
            & \le 
            \int_{T _\nu} ^T \intset{x _d = 0} {
                \nu \tomeganu \indtomeganu
            } \dx' \dt
            \\
            & \qquad + 
                \int_{T _\nu} ^T \intset{x _d = 0} 
                \maxs{
                    \frac\nu t, \frac{\nu ^2}{W ^2}, \frac{\nu ^2}{H ^2}
                }
                \dx' \dt.
        \end{aligned}
    \end{equation}
    For the first term in \eqref{eqn:split}, we apply \eqref{eqn:boundary-vorticity-d=2} and obtain
    \begin{equation}
        \label{eqn:split-first-d=2}
        \begin{aligned}
            & \int_{T _\nu} ^T \intset{x _d = 0} {
                A \nu \tomeganu \indtomeganu
            } \dx' \dt 
            \\
            & \qquad 
            \le \nmLpWT{\frac32, \infty}{\nu \tomeganu \indtomeganu} \nmLpWT{3, 1}{A}
            \\
            & \qquad 
            \le \frac18 \nu \nmLWT2{\grad \unu} ^2 + C A ^3 T \abs{\pOmega}.
        \end{aligned}    
    \end{equation}
    For the second term in \eqref{eqn:split}, it is bounded by
    \begin{equation*}
        \begin{aligned}
            & A \int_{T _\nu} ^T \intset{x _d = 0} 
            \maxs{
                \frac\nu t, \frac{\nu ^2}{W ^2}, \frac{\nu ^2}{H ^2}
            }
            \dx' \dt \\
            & \qquad \le A \int_{T _\nu} ^T \intset{x _d = 0} 
                \frac\nu t
            \dx' \dt + A \int_{T _\nu} ^T \intset{x _d = 0} 
            \maxs{
                \frac{\nu ^2}{W ^2}, \frac{\nu ^2}{H ^2}
            } 
            \dx' \dt \\
            & \qquad \le A \nu \log \pthf{T}{T _\nu} \vert\pOmega\vert
            + A \nu ^2 \mins{W, H} ^{-2} T \vert\pOmega\vert 
            \\
            & \qquad \le A ^2 \abs{\Omega} \Re \inv \log \pthf{4 T}{T _*}
            + A ^3 T \abs{\pOmega} \Re ^{-2} \maxs{H / W, 1} ^{2}.
        \end{aligned}
    \end{equation*}
    Since $\frac1{T _*} = C E ^2 \abs{\pOmega} ^{-2} \nu ^{-3} = C \pthf{E}{A ^2 \abs{\Omega}} ^2 \Re ^3 \frac AH$, we separate the log as
    \begin{equation*}
        \begin{aligned}
            \log \pthf{4T}{T _*} \le 3 \log \Re + 2 \pthf{E}{A ^2 \abs{\Omega}} + \frac{A T} H + C.
        \end{aligned}
    \end{equation*}
    Thus the second term in \eqref{eqn:split} is bounded by
    \begin{equation}
        \label{eqn:1/t-remainder}
        \begin{aligned}
            & A \int_{T _\nu} ^T \intset{x _d = 0} 
            \maxs{
                \frac\nu t, \frac{\nu ^2}{W ^2}, \frac{\nu ^2}{H ^2}
            }
            \dx' \dt \\
            & \qquad \le A ^2 \abs{\Omega} \Re \inv \log \pth{\Re + C} + 2 \Re \inv E \\
            & \qquad \qquad 
            + A ^3 T \abs{\pOmega} \left(
                \Re \inv + \Re ^{-2} \maxs{H / W, 1} ^{2}
            \right).
        \end{aligned} 
    \end{equation}

    Plugging \eqref{eqn:split-first-d=2}-\eqref{eqn:1/t-remainder} into \eqref{eqn:split} and applying to \eqref{eqn:combined} (naturally for $\set{x _d = H}$ the same estimate), we conclude for every $T > T _\nu$ that
    \begin{align*}
        &\nmLW2{\unu - \ub} ^2 (T) + \frac\nu2 \nor{\grad \unu} _{L ^2 ((T _\nu, T) \times \Omega)} ^2 \\
        &\qquad \le 4 \nmLW2{\unu - \ub} ^2 (0) + 2 G \int _{T _\nu} ^T \nmLW2{\unu - \ub} ^2 (t) \dt \\
        &\qquad \qquad + 2 \nu G ^2 T \vert\Omega\vert + A ^2 \abs{\Omega} \Re \inv \log \pth{\Re + C} + 2 \Re \inv E
        \\
        &\qquad \qquad
        + C A ^3 T \abs{\pOmega} \left(
            1 + \Re ^{-2} \maxs{H / W, 1} ^{2} 
        \right).
    \end{align*}
    Combined with \eqref{eqn:0-Tnu} we see indeed that the above inequality is true for any $T > 0$, so applying Gr\"onwall's inequality yields
    \begin{align*}
        & \sup _{0 \le t \le T} \set{
            \nmLW2{\unu - \ub} ^2 (t) + \frac\nu2 \nmLWt2{\grad \unu} ^2
        } \\
        & \le \exp(2 G T) \Biggl\{
            4 \nmLW2{\unu (0) - \ub} ^2 +
            C A ^3 T \abs{\pOmega} \left(
                1 + \Re ^{-2} \maxs{H / W, 1} ^{2}
            \right) + R _\nu
        \Biggr\},
    \end{align*}
    where the remainder terms $R _\nu$ is defined as
    \begin{align*}
        R _\nu &= 
        2 \nu G ^2 T \abs{\Omega} + A ^2 \abs{\Omega} \Re \inv \log \pth{\Re + C} + 2 \Re \inv E.
    \end{align*}
    Finally, if $\Re$ is sufficiently small, then the estimate holds true automatically by $\Re \inv E$ term according the trivial bound \eqref{eqn:trivial-bound}. Otherwise, by $\Re ^{-2} \le C$ and $\Re \inv \log (\Re + C) \le C \log (2 + \Re)$ we complete the proof.
\end{proof}

\begin{proof}[Proof of Theorem \ref{thm:constant-shear}.]
    In this particular setting, $G = 0$, $E = A ^2 \vert\Omega\vert$, $W / H = 1$. Therefore we can bound
    \begin{align*}
        R _\nu &\le C A ^2 \abs\Omega \Re ^{-1} \log (2 + \Re) + 2 \Re \inv E \le C A ^2 \abs\Omega \Re \inv \log (2 + \Re)
    \end{align*}
    which finishes the proof of the theorem.
\end{proof}





\appendix

\section{Construction of Weak Solutions to the Euler Equation with Layer Separation}
\label{appendix}

This appendix is dedicated to the proof of Proposition \ref{prop-convex}.
In \cite{Szekelyhidi2011}, Sz\'ekelyhidi constructed weak solutions to \eqref{eqn:euler} with strictly decreasing energy profile with vortex sheet initial data in a unit torus $\Omega = \T ^d$, by means of convex integration introduced in \cite{deLellis2010}.

We will first construct a weak (distributional) solution $(v, P)$ to \eqref{eqn:euler} in a two-dimensional set $\T \times (0, 1)$, such that $v = e _1$ at $t = 0$ and $\half \nmL2v ^2 (t) = \half - r t$ at a constant rate $r > 0$ for small $t$. 
To achieve this, we follow the ideas of \cite{Szekelyhidi2011}. However, we first construct a subsolution $\vb$ on a bigger domain $\tilde\Omega=\T \times[-1, 2]$, that we will convex integrate only on $\T\times(0,1)$. The result function $v$ is  a solution to \eqref{eqn:euler} only inside $\T \times (0, 1)$, but it keeps the global incompressibility $\div v = 0$ in $\T \times [-1, 2]$, together with $v = 0$ on $\T \times (-1, 0) \cup (1, 2)$. This provides the impermeability condition needed at the boundary. More precisely, consider $(\vb, \ub, \qb): (0, T) \times \tOmega \to \RR2 \times \mathcal S _0 ^{2 \times 2} \times \R$ with respect to some $\eb: (0, T) \times \tOmega \to [0, \infty)$, satisfying $\vb \in L _{\loc} ^2$, $\ub \in L _{\loc} ^1$, $\qb \in \mathcal D'$, and in the distribution sense 
\begin{align}
    \label{eqn:subsolution}
    \begin{cases}
        \pt \vb + \div \ub + \grad \qb = 0 \\
        \div \vb = 0
    \end{cases}
\end{align}
and almost everywhere 
\begin{align*}
    \vb \tensor \vb - \ub \le \eb \; \Id.
\end{align*}
Here $\mathcal S _0 ^{2 \times 2}$ is the space of trace-free two-by-two matrices.

To achieve this, we set 
\begin{align*}
    \vb &= (\alpha, 0), &
    \ub &= \begin{pmatrix}
        \beta & \gamma \\
        \gamma & -\beta 
    \end{pmatrix}, & 
    \qb &= \beta
\end{align*}
for some $\alpha (t, x _2), \beta (t, x _2), \gamma (t, x _2)$ to be fixed. With this choice, we need 
\begin{align*}
    \pt \alpha + \partial _{x _2} \gamma &= 0, &
    \begin{pmatrix}
        \eb - \alpha ^2 + \beta & \gamma \\
        \gamma & \eb -\beta 
    \end{pmatrix} \ge 0.
\end{align*}
The second constraint can be simplified to 
\begin{align*}
    2 \eb - \alpha ^2 &\ge 0, &
    (\eb - \alpha ^2 + \beta)(\eb - \beta) \ge \gamma ^2.
\end{align*}
Denote $\fb = \eb - \frac12 \alpha ^2$, $\delta = \beta - \frac12 \alpha ^2$, then 
\begin{align*}
    \begin{cases}
        \fb \ge 0 \\
        (\fb + \delta)(\fb - \delta) \ge \gamma ^2
    \end{cases} \Rightarrow \fb \ge \sqrt{\gamma ^2 + \delta ^2} \Rightarrow \eb \ge \half \alpha ^2 + \sqrt{\gamma ^2 + \delta ^2} \ge \half \alpha ^2 + \vert\gamma\vert,
\end{align*}
which will be the only constraint by setting $\beta = \half \alpha ^2$ thus $\delta = 0$.
It suffices to find $(\alpha, \gamma)$ that solves $\pt \alpha + \partial _{x _2} \gamma = 0$, i.e. we require the conservation of momentum and need
\begin{align*}
    \ddt \int \alpha \dx _2 &= 0, & 
    \gamma &= \int _{0.5} ^{x _2} \pt \alpha \dx _2, &
    \bar e &\ge \half \alpha ^2 + \vert\gamma\vert.
\end{align*}
Let us mimic the strategy in \cite{Szekelyhidi2011} and work with a different vortex-sheet initial data:
\begin{align*}
    \alpha (0, x _2) = \begin{cases}
        1 & 0 \le x _2 \le 1 \\
        0 & \text{otherwise}
    \end{cases}
\end{align*}
and let $\alpha (t, x _2)$ be the piecewise linear function interpolating $(-1, 0)$, $(0, 0)$, $(\lambda t, 1)$, $(1 - \lambda t, 1)$, $(1, 0)$, $(2, 0)$ for some fixed $\lambda > 0$ to be determined as in Figure \ref{fig:alpha}.
\begin{figure}[htbp]
    \label{fig:alpha}
    \centering
    \includegraphics{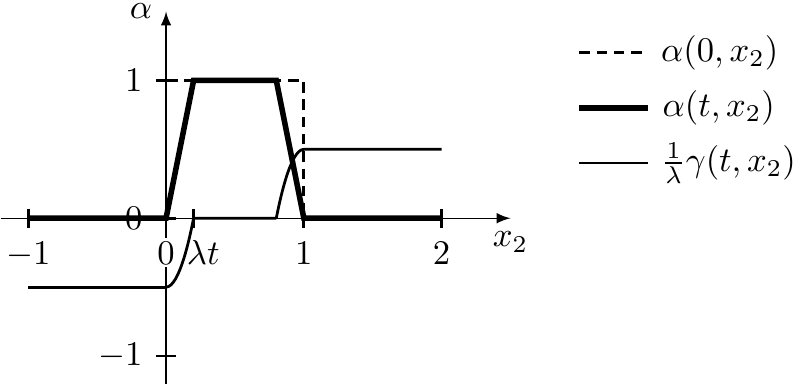}
    \caption{The graph of $\alpha (t, x _2), \frac1\lambda \gamma (t, x _2)$ for a fixed $0 \le t < T = \frac1{2\lambda}$}
\end{figure}

Under this setup, it is simple to see that
\begin{align*}
    \partial _{x _2} \gamma = -\pt \alpha = \lambda \alpha \vert\partial _{x _2} \alpha\vert,
\end{align*}
from which we can recover
\begin{align*}
    \gamma (t, x _2) = 
    \begin{cases}
        -\frac\lambda2 \pth{1 - \alpha ^2 (t, x _2)} & -1 \le x _2 \le \half \\
        \frac\lambda2 \pth{1 - \alpha ^2 (t, x _2)} & \text{otherwise}
    \end{cases}
\end{align*}
and as a consequence, we need 
\begin{align*}
    \eb \ge \half \alpha ^2 + \vert\gamma\vert = \half \alpha ^2 + \frac\lambda2 (1 - \alpha ^2) = \half - \half (1 - \lambda) (1 - \alpha ^2).
\end{align*}
Let us fix $\lambda, \e \in (0, 1)$, and set 
\begin{align}
    \label{eqn:energy-density}
    \bar e = \half - \frac\e2 (1 - \lambda) (1 - \alpha ^2).
\end{align}
Then $\eb > \half \alpha ^2 + \vert\gamma\vert$ in the space-time region $\mathcal U := (0, T) \times \T \times(0, 1) \cap \set{\alpha < 1}$.

We are now ready to apply Theorem 1.3 of \cite{Szekelyhidi2011} when convex integrating in $(0,T) \times \T \times (0, 1)$ only.  This provides  infinitely many $(\tilde v, \tilde u) \in L _{loc} ^\infty ((0, T) \times \tOmega)$ with $\tilde v \in C (0,T; L _{\text{weak}} ^2 (\tOmega)$ such that $(\tilde v, \tilde u, 0)$ satisfies \eqref{eqn:subsolution}, $(\tilde v, \tilde u) = 0$ a.e. in $\mathcal U ^c = (0, T) \times \T \times ((-1, 0) \cup (1, 2)) \cup \{\alpha = 1\}$, and $v := \vb + \tilde v$, $u := \ub + \tilde u$ satisfy
\[
    v \tensor v - u = \eb \; \Id \qquad \ae \inn \mathcal (0, T) \times \T \times (0, 1).
\]
From the second equation of \eqref{eqn:subsolution},  $\partial_{x_2} v_2=-\partial_{x_1}v_1$, and $v_2\in C_{x_2}(W^{-1,\infty}_{x_1})$. But since we didn't convex integrate on $(0,T)\times\T\times ((-1,0)\cup(0,1))$, we still have $v_2=0$ at $x_2=0$ and $x_2=1$. This provides the impermeability boundary conditions at these points.

Then $(v, P)$ satisfies \eqref{eqn:euler} with the impermeability conditions in $(0,T)\times\T\times(0,1)$ in the distributional sense for $P = \qb - \eb$, and $\half \vert v\vert ^2 = \eb$ matches the energy density profile given in \eqref{eqn:energy-density} (note that the constructed solution is not solution to \eqref{eqn:euler} in the domain $(0,T)\times \T\times (-1,2)$). Now, we have on  $(0,T)\times\T\times(0,1)$:
\[
    \ddt \int \hfsq v \dx = \e (1 - \lambda) \int \alpha \pt \alpha \dx _2 = -\e \lambda (1 - \lambda) \int \alpha ^2 \vert\partial _{x _2} \alpha\vert \dx _2 = -\frac23 \e \lambda (1 - \lambda),
\]
i.e. $\half \nmL2v ^2$ decreases linearly at rate $r := \frac23 \e \lambda (1 - \lambda)$. 

We consider the deviation from initial value. Since $\tilde v = 0$ a.e. at $t = 0$, we know $v (0) = \vb (0) = \pm e _1$, and
\begin{align*}
    \half \ddt \int \vert v (t) - v (0)\vert ^2 \dx &= \ddt \int \hfsq{v (t)} \dx - \ddt \int v (t) \cdot v (0) \dx \\
    &= - r - \int \pt v (t) \cdot v (0) \dx \\
    &= - r + \int \div u (t) \cdot v (0) \dx.
\end{align*}
The quantity $\eb$ and $\qb$ depend only on $t, x_2$, so the equation on $v_1$ from \eqref{eqn:subsolution} has no pressure and verify:
$$
\partial _{x_2} u_{12} = -\pt v _1 - \partial_{x_1} u_{11}. 
$$
Especially, $u_{12}\in C_{x_2}(W^{-1,\infty}_{t,x_1})$.  
Therefore,
\begin{align*}
    \int \div u (t) \cdot v (0) \dx &= \int _{\T} - u _{12} (t, x _1, 0) + u _{12} (t, x _1, 1) \dx _1 \\
    &= \int _{\T} - \ub _{12} (t, x _1, 0) + \ub _{12} (t, x _1, 1) \dx _1 \\
    &= - \gamma (t, 0) + \gamma (t, 1) = \lambda.
\end{align*}
This gives 
\begin{align*}
    \half \ddt \int \vert v (t) - v (0)\vert ^2 \dx = \lambda - r = \lambda - \frac23 \e \lambda (1 - \lambda).
\end{align*}
This rate converges to 1 by setting $\lambda \to 1$ and $\e \to 0$, thus
\begin{align*}
    \half \nor{v (t) - e _1} _{L ^2 (\T \times [0, 1])} ^2 = C t, \qquad \forall t \in \left(0, \frac1{2 \lambda}\right).
\end{align*}
Moreover, $v = 0$ on $\set{x _2 = 0, 1}$.

Now for some $A > 0$, define $(v ^*, P ^*): (0, \frac1{2 \lambda A}) \times \Omega \to \R ^2 \times \R$ by time rescaling $v ^* (t, x) = A v (A t, x)$, $P ^* (t, x) = A ^2 P (A t, x)$, where $\Omega = \T \times [0, 1]$ is the unit channel. Then $v ^* (0) = A e _1$ in $\Omega$, $v ^* (t) = 0$ on $\partial \Omega$ and 
\begin{align*}
    \half \nor{v (t) - A e _1} _{L ^2 (\T \times [0, 1])} ^2 = C A ^3 t, \qquad \forall t \in \left(0, \frac1{2 \lambda A}\right)
\end{align*}
for some $C, \lambda$ satisfying $0 < C < \lambda < 1$.

\bibliographystyle{alpha}
\bibliography{ref.bib}

\end{document}